\documentclass[11pt]{amsart}

\usepackage{enumerate, amssymb}
\usepackage[usenames, dvipsnames]{color}
\usepackage[margin=1in]{geometry} 
\usepackage{tikz}
\usetikzlibrary{matrix,arrows}

\numberwithin{equation}{section}
\newtheorem{theorem}[equation]{Theorem}
\newtheorem*{thm}{Theorem}
\newtheorem{proposition}[equation]{Proposition}
\newtheorem{lemma}[equation]{Lemma}
\newtheorem{corollary}[equation]{Corollary}

\newtheorem{assumption}[equation]{Assumption}

\theoremstyle{definition}
\newtheorem{rmk}[equation]{Remark}
\newenvironment{remark}[1][]{\begin{rmk}[#1] \pushQED{\qed}}{\popQED \end{rmk}}
\newtheorem{eg}[equation]{Example}

\newtheorem{defn}[equation]{Definition}
\newenvironment{definition}[1][]{\begin{defn}[#1]\pushQED{\qed}}{\popQED \end{defn}}
\newtheorem{ques}[equation]{Question}
\newenvironment{question}[1][]{\begin{ques}[#1]\pushQED{\qed}}{\popQED \end{ques}}

\usepackage[colorlinks=true, pdfstartview=FitV, linkcolor=black, citecolor=black, urlcolor=black]{hyperref}

\newcommand{\cB}{\mathcal{B}}
\newcommand{\cN}{\mathcal{N}}
\newcommand{\cQ}{\mathcal{Q}}
\newcommand{\cS}{\mathcal{S}}
\newcommand{\cT}{\mathcal{T}}

\newcommand{\fb}{\mathfrak{b}}

\newcommand{\ze}{\ensuremath{\epsilon}}

\newcommand{\zl}{\ensuremath{\lambda}}

\newcommand{\ZZ}{\mathbb{Z}}

\newcommand{\sqbinom}[2]{\genfrac{[}{]}{0pt}{}{#1}{#2}}

\newcommand{\setst}[2]{\left\{ #1 \mid #2 \right\}}

\newcommand{\bimod}{\mathsf{bimod}}

\newcommand{\comment}[1]{}

\newcommand{\Hom}{\operatorname{Hom}}
\DeclareMathOperator{\rad}{rad} 
\DeclareMathOperator{\End}{End}
\DeclareMathOperator{\nrep}{nlrep}

\newcommand{\xto}[1]{\xrightarrow{#1}}

\usepackage{tikz-cd}

\newcommand{\kk}{\Bbbk}

\newcommand{\uqb}{U_q(\mathfrak{b})}
\newcommand{\Uqsl}{U_q(\mathfrak{sl}_2)}
\newcommand{\uqsl}{u_q(\mathfrak{sl}_2)}
\newcommand{\rep}{\mathsf{rep}}
\DeclareMathOperator{\lcm}{lcm}
\renewcommand{\mod}{\operatorname{mod}}

\newtheorem{notation}[equation]{Notation}

\begin{document}
\title{Hopf actions of some quantum groups on path algebras}

\author{Ryan Kinser}
\address{University of Iowa, Department of Mathematics, Iowa City, USA}
\email[Ryan Kinser]{ryan-kinser@uiowa.edu}

\author{Amrei Oswald}
\address{University of Iowa, Department of Mathematics, Iowa City, USA}
\email[Amrei Oswald]{amrei-oswald@uiowa.edu}

\begin{abstract}
Our first collection of results parametrize (filtered) actions of a quantum Borel $\uqb \subset \Uqsl$ on the path algebra of an arbitrary (finite) quiver.
When $q$ is a root of unity, we give necessary and sufficient conditions for these actions to factor through corresponding finite-dimensional quotients, generalized Taft algebras $T(r,n)$ and small quantum groups $\uqsl$.

In the second part of the paper, we shift to the language of tensor categories. 
Here we consider a quiver path algebra equipped with an action of a Hopf algebra $H$ to be a tensor algebra in the tensor category of representations $H$.
Such a tensor algebra is generated by an algebra and bimodule in this tensor category.
Our second collection of results describe the corresponding bimodule categories via an equivalence with categories of representations of certain explicitly described quivers with relations.
\end{abstract}

\makeatletter
\@namedef{subjclassname@2020}{%
  \textup{2020} Mathematics Subject Classification}
\makeatother



\maketitle

\setcounter{tocdepth}{1}
\tableofcontents

\section{Introduction}
\subsection{Context and motivation}
Throughout the paper we work over a field $\kk$ which is often omitted from the notation.  The only assumptions on the field are that it must contain the necessary roots of unity whenever they are referenced.

The classical mathematical notion of ``symmetry'' can be formalized by group actions.  A group can act directly on an object itself, or indirectly on various spaces of functions on an object.  
An advantage of the latter approach is that spaces of $\kk$-valued functions on any object (the functions may be restricted to satisfy some reasonable property) are $\kk$-vector spaces, allowing for the introduction of tools from linear algebra and representation theory in the study of symmetry.
This naturally leads from groups to the study of group algebras and their actions on other $\kk$-algebras.

More generally, actions of Hopf algebras on $\kk$-algebras are one way to formalize the mathematical notion of \emph{quantum symmetry}.
One may see the following (highly incomplete) list of recent works as entry points to the extensive literature on the topic
\cite{CEW16,CKWZ16,EW16,EW17,CKWZ1,CKWZ2,Cline19,BHZ19,Negron19,CG20,EKW20,DNN20,LNY20,CY20}, and the survey articles \cite{Kirkman16,Walton2019}.
In this paper we study actions of several families of Hopf algebras:
\begin{equation}\label{eq:algebras}
\uqb, \quad \Uqsl, \quad T(n,r), \quad \uqsl.
\end{equation}
on path algebras of quivers. 
Here, the first two are rank 1 quantum groups (deformed enveloping algebras), and the latter two are certain finite-dimensional quotients of these; see Section \ref{sec:algebradefs} for precise definitions.
Our work in this paper significantly extends the main results of \cite{KW16} (and a subsequent partial generalization in one chapter of \cite{Berrizbeitia18}), while simultaneously connecting with the tensor category framework for studying Hopf actions on quiver path algebras introduced in \cite{EKW20}.

\subsection{Summary of main results}
Below, $H$ always refers to one of the algebras in \eqref{eq:algebras}.
Our main results describe all filtered actions (see Assumption \ref{assume}) of such $H$ on path algebras of arbitrary quivers with finitely many vertices and arrows, in two different ways.

In Section \ref{sec:parametrize}, we parametrize these actions by basic linear algebraic data.  We paraphrase Theorems \ref{thm:uqbonkQ} and \ref{thm:UqslonkQ} as follows, where the ``certain conditions'' mentioned below are explicitly listed in these theorems.
\begin{thm}
The following data determines a (filtered) Hopf action of $\uqb$ (resp., $\Uqsl$) on a path algebra $\kk Q$, and all such actions are of this form:
\begin{enumerate}[(i)]
    \item a Hopf action of $\uqb$ (resp., $\Uqsl$) on $\kk Q_0$, as described in Proposition \ref{prop:uqbonvertices} (resp. \ref{prop:uqslonvertices});
    \item a representation of $G$ on $\kk Q_1$ which is compatible with its $\kk Q_0$-bimodule structure;
    \item a linear endomorphism (resp. pair of linear endomorphisms) of $\kk Q_0 \oplus \kk Q_1$ 
    satisfying certain conditions.
\end{enumerate}
\end{thm}
The linear endomorphism(s) mentioned in \textit{(iii)} are derived from the actions of the standard skew-primitive generators of $\uqb$ and $\Uqsl$,
but are in some sense independent of Hopf action on $\kk Q_0$ in \textit{(i)}.
Corollaries \ref{cor:taftquivers} and \ref{cor:uqslonkQ2} give easily verifiable criteria for an action above to factor through the finite dimensional quotients $T(r,n)$ and $\uqsl$.

In Section \ref{sec:bimodules}, we shift perspective following recent joint work of the first author with Etingof and Walton \cite{EKW20},
viewing a path algebra with an action of $H$ as particular case of a tensor algebra in the tensor category $\rep(H)$.
This leads to an analysis of bimodule categories over commutative algebras in $\rep(H)$ for each Hopf algebra $H$ in \eqref{eq:algebras}, leveraging our work in Section \ref{sec:parametrize}.
Our main results here give describe the structure of these bimodule categories.
We provide a brief outline below; see Theorems \ref{thm:uqbequivalence} and \ref{thm:uqslequivalence} and their corollaries for the precise statements.

\begin{thm} 
Consider the tensor category $\mathcal{C}=\rep(H)$ 
and a pair of indecomposable algebras $S,\, S'$ in $\mathcal{C}$ such that $S=\kk^m$ and $S'=\kk^{m'}$ as $\kk$-algebras.
Then the bimodule category $\bimod_{\mathcal{C}}(S, S')$ is equivalent to the following category of representations in each case:
\begin{enumerate}[(a)]
\item for $H=\uqb$, representations of $\Gamma(q^\ell,d)$ which are nilpotent on loops (see Notation \ref{not:uqb})
\item for $H=T(r,n)$, representations of $\Gamma_T$ (see Notation \ref{not:T(r,n)})
\item for $H=\Uqsl$, representations of $\Gamma'(q^{2\ell},q^2,d)$ which are nilpotent on loops (see Notation \ref{not:Uqsl2})
\item for $H=\uqsl$, representations of $\Gamma'_T$ (see Notation \ref{not:uqsl2}).
\end{enumerate}
Each algebra above is presented as a quotient of a path algebra (via explicitly given relations) of a quiver whose connected components are of the form shown in \eqref{eq:Sinfinity}, \eqref{eq:Sr}, \eqref{eq:S'3}, and \eqref{eq:T'}.
\end{thm}

One immediate consequence of these equivalences is that the bimodule category in each case is independent of the specific $H$-actions on $S, S'$, since the corresponding algebras are (Corollary \ref{cor:independent}).
Another consequence is that classifying $H$-actions even on very small quivers will usually be 
``at least as complicated as'' classification of representations of the free associative $\kk$-algebra $\kk\langle x,y\rangle$, or equivalently, pairs of matrices up to simultaneous conjugation.
This is because the algebras appearing in our theorems of Section \ref{sec:bimodules}
typically have finite-dimensional algebras of wild representation type as quotients.
It seems to us that classification up to Morita equivalence in $\rep(H)$ should also be difficult, but we do not see how to make that precise.
A notable exception is when $H=T(n,r)$ is a generalized Taft algebra, in which case our results show that the relevant bimodule categories can actually be of finite representation type, as they can be equivalent to representation categories of Nakayama algebras.

\subsection*{Acknowledgements}
The authors thank Chelsea Walton for valuable feedback on the first draft of this manuscript.

\section{Background and preliminaries}
\subsection{Hopf algebras and their actions}\label{sec:Hopfactions}
We refer the reader to standard texts such as \cite{Montgomery,DNR,Radford} for background on Hopf algebras and Hopf actions, and we use standard notation such as $\Delta$ for the comultiplication and $\varepsilon$ for the counit.
The term \emph{algebra} in this paper always refers to an associative $\kk$-algebra with identity.

Recall that given a Hopf algebra $H$ and an algebra $A$,  a \emph{(left Hopf) action of $H$ on $A$} consists of a left $H$-module structure on $A$ satisfying:
\begin{enumerate}[(a)]
\item $ h \cdot (pq) = \sum_i (h_{i,1} \cdot p)(h_{i,2} \cdot q)$ for all $h\in H$ and $p, q \in A$, where $\Delta(h)=\sum_i h_{i,1}\otimes h_{i,2}$, and
\item $h \cdot 1_A = \varepsilon(h)1_A$ for all $h \in H$.
\end{enumerate}
In this case we also say that $A$ is a \emph{left $H$-module algebra}.
This is equivalent to $A$ being an algebra in the tensor category $\rep(H)$, the category of finite-dimensional representations of $H$.

\subsection{Generalized Taft algebras and related quantum groups}\label{sec:algebradefs}
In this section we define the specific Hopf algebras whose actions are studied in this paper.  We omit mention of the counits since the conditions they impose throughout the paper are always trivial to check.
The symbol $q$ always denotes an element of $\kk$, with further restrictions depending on the algebra. 

\begin{definition}\label{def:taft}
Let $q \in \kk$ with $q \neq 0, \pm 1$. Following \cite[Ex.~1.5]{majidbook}, the Hopf algebra $\uqb$ has $\kk$-algebra generators $x, g, g^{-1}$
and relations
\begin{equation}
gg^{-1}=1=g^{-1}g, \quad gxg^{-1}=q x .
\end{equation}
The comultiplication is given by
\begin{equation}
\Delta(g)=g\otimes g, \quad \Delta(x)=1\otimes x + x \otimes g.
\end{equation}

If $q$ is a primitive $r^{\rm th}$ root of unity and $n$ a positive integer multiple of $r$, then the \emph{generalized Taft algebra} \cite{Radford75} is the Hopf quotient $T(r,n):=\uqb/\langle g^n-1,\, x^r\rangle$ (see also \cite{cibils93}).  We write $T(n):=T(n,n)$ for short; these are the classical \emph{Taft algebras} \cite{Taft71}.
\end{definition}

The Hopf algebra $\uqb$ can be viewed as a deformation of the universal enveloping algebra of $\mathfrak{b}$. We now define a deformation of the universal enveloping algebra of $\mathfrak{sl}_2$, and a certain finite-dimensional Hopf quotient.  

\begin{definition}
\label{def:uqsl2}
Following  \cite[\S7]{majidbook}, the algebra $\Uqsl$ is the $\kk$-algebra with generators $E,F, K,K^{-1}$ subject to relations
\begin{equation}\label{eq:uqslrelations}
KEK^{-1}=q^2 E,\quad KFK^{-1}=q^{-2}F, \quad [E,F]=\frac{K-K^{-1}}{q-q^{-1}}
\end{equation}
along with $KK^{-1}=1=K^{-1}K$.
The comultiplication is given by
\begin{equation}
\Delta(E)=1 \otimes E+E\otimes K, \quad \Delta(F)=K^{-1}\otimes F+F\otimes 1, \quad \Delta(K)=K \otimes K.
\end{equation}

When $q$ is a primitive $n^{\rm th}$ root of unity with $n>2$ and $n$ odd, we also have the \emph{small quantum group} $u_q(\mathfrak{sl}_2)$, or \emph{Frobenius-Lusztig kernel}, as the Hopf quotient $U_q(\mathfrak{sl}_2)/\langle K^n-1,\ E^n, F^n \rangle$ \cite[Ex.~6.4]{majidbook}.
\end{definition}

The algebras $\uqb$ and $T(n)$ are analogues of Borel subalgebras in $U_q(\mathfrak{sl}_2)$ and $u_q(\mathfrak{sl}_2)$, respectively.  
Namely, we have isomorphisms of Hopf algebras
\begin{equation}\label{eq:borels}
\begin{split}
\langle E, K \rangle \simeq U_{q^2}(\fb)\qquad  &\text{ where }K \leftrightarrow g, \quad E \leftrightarrow x\\
\langle F, K \rangle \simeq U_{q^{-2}}(\fb)\qquad &\text{ where }K \leftrightarrow g, \quad F \leftrightarrow g^{-1}x,
\end{split}
\end{equation}
and similarly $\uqsl$ has two subalgebras isomorphic to Taft algebras.

\subsection{Skew-primitives and $q$-integers}
Following \cite[\S7]{majidbook}, for $q \neq 0,\pm 1$ we define the quantum integers and binomial coefficients for integers $0 < m < n$ as
\begin{equation}
[m]_q = \frac{1-q^m}{1-q}, \qquad \sqbinom{n}{m}_{q} = \frac{[n]_q!}{[m]_q![n-m]_q!},
\end{equation}
where $[k]_q! = [k]_q [k-1]_q \cdots [2]_q [1]_q$ for $k \in \ZZ_{>0}$.  
We observe that for $0 < m < n$ we have
\begin{equation}\label{eq:binomvanish}
q^n=1 \quad \Longrightarrow \quad \sqbinom{n}{m}_{q} = 0.
\end{equation}
By convention we 
interpret $\sqbinom{n}{m}_{q} =1$ when $m=0$ or $n$.

Let $k, l$ be positive integers.  The set of \emph{weak compositions} of $k$ of length $l$ is
\begin{equation}
WC(k,l)=\left\{\lambda=(\lambda_1,\lambda_2,\ldots,\lambda_l)\in \ZZ_{\geq 0}^l \bigg\vert \sum\limits_{i=1}^l \lambda_i = k \right\},
\end{equation}
and for $\lambda \in WC(k,l)$, we write $|\lambda|=k$.  We denote by $e_i \in WC(1,l)$ the standard basis vector which is 1 in coordinate $i$ and 0 elsewhere.  For $\lambda \in WC(k,l)$, the $q$-multinomial coefficient is defined as
\begin{equation}
\sqbinom{k}{\lambda}_{q} = \frac{[k]_q!}{[\lambda_1]_q! \cdots [\lambda_l]_q!}.
\end{equation}
Partial sums of the entries of $\lambda \in WC(k,l)$ frequently appear in our formulas, so we define 
\begin{equation}
\lambda^i = \lambda_1 + \lambda_2 + \cdots + \lambda_{i-1}.
\end{equation}
(and take $\lambda^1 = 0$).
These satisfy the following recursion, which is a $q$-multinomial analogue of Pascal's formula for binomial coefficients:
\begin{equation}\label{eq:qmultrec}
\sqbinom{k}{\lambda}_q= \sum\limits_{i=1}^{k} q^{\lambda^i} \sqbinom{k-1}{\lambda-e_i}_q \qquad \text{for } \lambda \in WC(k,l).
\end{equation}

\begin{notation}
Let $G,X \in H$ be two elements of an algebra $H$ and $\lambda \in WC(k,l+1)$.  We define the element $G\otimes^{\lambda} X$ of $H^{\otimes l+1}$ by
\begin{equation}\label{eq:GlX}
\begin{split}
G\otimes^{\lambda} X
&:=X^{\lambda_1}\otimes G^{\lambda_1} X^{\lambda_2}\otimes G^{\lambda_1+\lambda_2}X^{\lambda_3} \otimes \cdots \otimes G^{\lambda_1+\cdots+\lambda_l}X^{\lambda_{l+1}}\\
&= G^{\lambda^1} X^{\lambda_1}\otimes G^{\lambda^2} X^{\lambda_2}\otimes G^{\lambda^3}X^{\lambda_3} \otimes \cdots \otimes G^{\lambda^{l+1}}X^{\lambda_{l+1}}.
\end{split}
\end{equation}
\end{notation}
For example, we have for each $1 \leq i \leq l$ and standard basis vector $e_i$ the element
\begin{equation}\label{eq:GeiX}
G \otimes^{e_i} X:= 1\otimes \cdots \otimes 1\otimes X \otimes G \otimes \cdots \otimes G
\end{equation}
where $X$ appears in factor $i$.

\begin{lemma}\label{lem:GXeilambda}
Let $G,X$ be elements of an algebra such that $GX=qXG$ for some $q \in \kk^\times$, and let $\lambda \in WC(k,l+1)$.  Then we have
\[
(G \otimes^{e_i} X)(G\otimes^\lambda X)=q^{-\lambda^i}\ G\otimes^{e_i+\lambda} X.
\]
\end{lemma}
\begin{proof} Recalling the notation in \eqref{eq:GlX} and \eqref{eq:GeiX}, we have
\begin{equation}
(G \otimes^{e_i} X)(G\otimes^\lambda X) 
=G^{\lambda^1}X^{\lambda_1}\otimes G^{\lambda^2} X^{\lambda_2}\otimes \cdots \otimes XG^{\lambda^i}X^{\lambda_i} \otimes G^{1+\lambda^{i+1}}X^{\lambda_{i+1}} \otimes \cdots \otimes G^{1+\lambda^{l+1}}X^{\lambda_{l+1}}.
\end{equation}
Then applying $XG^{\lambda^i}= q^{-\lambda^i}G^{\lambda^i}X$ in the $i^{\rm th}$ tensor factor and pulling the resulting $q^{-\lambda^i}$ to the front, 
we arrive at
\begin{equation}
q^{-\lambda^i}\left(G^{\lambda^1}X^{\lambda_1}\otimes G^{\lambda^2} X^{\lambda_2}\otimes \cdots \otimes G^{\lambda^i}X^{1+\lambda_i} \otimes G^{1+\lambda^{i+1}}X^{\lambda_{i+1}} \otimes \cdots \otimes G^{1+\lambda^{l+1}}X^{\lambda_{l+1}}\right)
\end{equation}
which is equal to $q^{-\lambda^i}\,G\otimes^{e_i+\lambda} X$ by definition.
\end{proof}

We prove a general combinatorial formula which can be applied in our setting.  Recall that $\Delta^l \colon H \to H^{\otimes l+1}$ is the map obtained by $l$ applications of $\Delta$, which is well defined by the coassociativity axiom of a Hopf algebra.

\begin{proposition}\label{prop:deltalxk}
Let $X$ be a $(1,G)$-skew primitive element in a Hopf algebra such that $GX=qXG$ for some $q \in \kk^\times$.  For any two integers $l,k \ge 1$, we have in $H^{\otimes l +1}$:
\begin{equation}\label{eq:deltalxk}
	\Delta^{l}(X^k) =\sum_{\lambda \in WC(k,l+1)} \sqbinom{k}{\lambda}_{q^{-1}} G \otimes^\lambda X
\end{equation}
\end{proposition}

	\begin{proof}
We prove the statement by induction on $k$ (for $l$ arbitrary).  For $k=1$, it follows easily from skew-primitivity of $X$ and the notation \eqref{eq:GeiX} that
\begin{align*}
\Delta^l(X) =  \sum_{i=1}^{l+1} G \otimes^{e_i} X
\end{align*}
Since $WC(1,l+1) = \{e_i\}_{i=1}^{l+1}$ and the binomial coefficients are all equal to 1, this confirms the base case.

Now assume that the formula \eqref{eq:deltalxk} is true for some $k \geq 1$.  Since $\Delta$ is an algebra morphism, we have by the induction hypothesis that
\begin{equation}
\begin{split}
\Delta^l(X^{k+1})&=\Delta^l(X)\Delta^l(X^k) = \left( \sum_{i=1}^{l+1} G \otimes^{e_i} X\right)
\left( \sum_{\lambda \in WC(k,l+1)} \sqbinom{k}{\lambda}_{q^{-1}} G \otimes^\lambda X \right)\\
&=\sum_{i=1}^{l+1}\sum_{\lambda \in WC(k,l+1)} \sqbinom{k}{\lambda}_{q^{-1}} (G \otimes^{e_i} X)(G \otimes^\lambda X)\\
&=\sum_{i=1}^{l+1}\sum_{\lambda \in WC(k,l+1)} \sqbinom{k}{\lambda}_{q^{-1}} q^{-\lambda^i} \ G \otimes^{e_i+\lambda} X
\end{split}
\end{equation}
where the last equality uses Lemma \ref{lem:GXeilambda}.
For each fixed $\mu \in WC(k+1,l+1)$, we group together all pairs in the sum above $(i, \lambda)$ such that $e_i+\lambda=\mu$, noting that this partitions the summands above.
The coefficient of $G\otimes^\mu X$ is then
\begin{equation}
\sum_{\substack{(i,\lambda)\\ e_i+\lambda=\mu}} \sqbinom{k}{\lambda}_{q^{-1}} q^{-\lambda^i}
= \sqbinom{k+1}{\mu}_{q^{-1}}
\end{equation}
where the last equality is \eqref{eq:qmultrec}.  This completes the induction step and the proposition is proven.
\end{proof}

\subsection{Quivers, path algebras, and their representations}\label{sec:quivers}
In this section, we establish notation and terminology for quivers and path algebras.  More detailed introductions to the topic can be found in textbooks such as \cite{ASS06,Schifflerbook,DWbook}.
Recall that a \emph{quiver} $Q =(Q_0, Q_1, s, t)$ is a quadruple consisting of a finite set of \emph{vertices} $Q_0$, a finite set of \emph{arrows} $Q_1$, and two functions $s, t : Q_1 \to Q_0$ producing the \emph{source} and \emph{target} of each arrow, respectively. We may visualize the arrows as follows:
\begin{center}
\begin{tikzcd}
\hspace{2mm} s(a) \arrow{r}{a} & t(a). \hfill \qedhere
\end{tikzcd}
\end{center}
We omit the parentheses when possible to lighten the notation, writing $sa$ and $ta$.

The \emph{path algebra} $\kk Q$ is the associative $\kk$-algebra whose underlying $\kk$-vector space has as its basis the set of all paths in $Q$, with the product of two paths being concatenation whenever possible (read left to right), and zero otherwise.	Therefore, $\kk Q$ has a direct sum decomposition as a vector space
\[
\kk Q= \bigoplus_{l=0}^\infty \kk Q_l,
\]
where $\kk Q_l$ is the subspace of $\kk Q$ generated by the set $Q_l$ of all paths of length $l$.
This decomposition makes $\kk Q$ a graded $\kk$-algebra.
For $i \in Q_0$, we write $e_i \in \kk Q$ for the idempotent corresponding to the length 0 path at $i$.

We note that $\kk Q_0$ is a semisimple $\kk$-algebra and $\kk Q_1$ a $\kk Q_0$-bimodule, inducing a natural isomorphism of the path algebra $\kk Q$ with the tensor algebra $T_{\kk Q_0}(\kk Q_1)$.
Put another way, $\kk Q$ is isomorphic to the quotient of the free associative $\kk$-algebra on generators $\{e_i\}_{i \in Q_0} \cup Q_1$, by the relations:
\begin{equation}\label{eq:kQrelations}
e_i e_j = \delta_{i,j}e_i \text{ for all } i \in Q_0, \qquad
 e_{sa}a = a = ae_{ta} \text{ for all } a \in Q_1.
 \end{equation}

We use the following slightly unconventional definition, with motivation explained below.
A \textit{representation} $V$ of a quiver $Q$ is an assignment of a finite dimensional vector space $V_i$ to each vertex $i \in Q_0$ along with a linear map $V_a: V_{ta} \to V_{sa}$ for each arrow $a \in Q_1$ (note that the map goes in the opposite direction of the arrow).
Equivalently, a representation of $Q$ is a finite dimensional right $\kk Q$-module, or a contravariant functor from the free category on $Q$ to the category of finite-dimensional vector spaces.
The reason for this convention is that we want to keep the same definition of path algebra as in \cite{KW16} which reads paths from left to right,
while using the standard convention for composition of linear maps (reading right to left) in a representation.

Given two representations $V$ and $W$ of $Q$, a \textit{morphism of quiver representations} $\varphi: V \to W$ is an assignment of a linear map 
$\varphi(i): V_i \to W_i$ to each $i \in Q_0$ so that $\varphi(sa)\circ V_a = W_a\circ\varphi(ta)$.

\subsection{Hopf actions on path algebras} \label{sec:Hopfactionpathalgebra}
We note once and for all that we restrict to actions satisfying the following assumption throughout the paper.
Let $H$ be a Hopf algebra, $Q$ a quiver, and set $S=\kk Q_0$ and $V=\kk Q_1$, so that $\kk Q=T_S(V)$.

\begin{assumption}\label{assume}
We assume that Hopf actions on path algebras in this paper preserve the ascending filtration by path length:
\begin{equation}
    S \subset S\oplus V \subset S \oplus V \oplus (V \otimes V) \subset \cdots.
\end{equation}
\end{assumption}

Under this assumption, we make the following elementary observations on parametrization of Hopf actions of $H$ on $T_S(V)$.  Let $G:=G(H)$ be the group of grouplike elements of $H$ below.

\begin{lemma}\label{lem:gactionvertices}

\begin{enumerate}[(i)]
\item Suppose an action of $H$ on $T_S(V)$ is given.  Then this action is completely determined by the $H$-module structure of $S \oplus V$.
\item Conversely, any $H$-module structure on $S \oplus V$ uniquely extends to an $H$-module structure on the algebra $T_S(V)$, which is a Hopf action exactly when it satisfies axioms (a) and (b) of Section \ref{sec:Hopfactions}. 
\item Any action of $G$ on $S$ is induced by a permutation action of $G$ on $Q_0$.
\end{enumerate}
\end{lemma}
\begin{proof}
Parts (i) and (ii) follow from the fact that $S, V$ generate $T_S(V)$ as an algebra.  
Part (iii) follows from the fact that a commutative semisimple algebra has a unique complete set of orthogonal idempotents, and an algebra automorphism must permute this set.
\end{proof}

\section{Parametrizing actions by linear algebraic data}\label{sec:parametrize}
In this section we parametrize actions of the algebras \eqref{eq:algebras} on path algebras of quivers, in terms if basic linear algebraic data.

\subsection{Actions of $\uqb$ and Taft algebras on products of $\kk$}
We first consider quivers without arrows (i.e. algebras of the form $\kk \times \kk \times \cdots \times \kk$).
The following result and its corollaries parametrize actions of $\uqb$ on such algebras.
We write $G=\langle g \rangle \cong \ZZ$, and  $|q|$ denotes the multiplicative order of $q$ in $\kk^\times$.

\begin{proposition}
\label{prop:uqbonvertices}
Let $Q_0$ be the vertex set of a quiver.  

\noindent (a) The following data determines a Hopf action of $\uqb$ on $\kk Q_0$.
\begin{enumerate}[(i)]
\item A permutation action of $G$ on the set $Q_0$;
\item A collection of scalars $(\gamma_i \in \kk)_{i \in Q_0}$ such that
\begin{equation}\label{eq:qgamma}
\gamma_{g\cdot i} = q^{-1} \gamma_i \quad \forall i \in Q_0.
\end{equation}
\end{enumerate}
The $x$-action is given by
\begin{equation} \label{eq:xonvertex}
x \cdot e_i = \gamma_i e_i - \gamma_i q^{-1} e_{g\cdot i} \quad {\text{ for all $i \in Q_0$.}}
\end{equation}

\noindent (b) Every action of $\uqb$ on $\kk Q_0$ is of the form above.
\end{proposition}
\begin{proof}
(A)
It is immediate to check that the data given in (i) and (ii) makes $\kk Q_0$ into a $\uqb$-module. To check we have a Hopf action, $g$ is grouplike and acts as an algebra automorphism by assumption, so it remains to check the action of $x$ satisfies the axiom of a Hopf action.  It is enough to consider the relations defining $\kk Q$ in \eqref{eq:kQrelations}, so we need to check that
\begin{equation}\label{eq:xdoteicheck}
x\cdot (\delta_{i,j} e_i) = e_i(x\cdot e_j) + (x\cdot e_i)(g\cdot e_j).
\end{equation}
Assume first $i=j$.  By substitution we get
\[
\begin{split}
\gamma_i e_i -\gamma_i q^{-1} e_{g\cdot i}&= e_i(\gamma_i e_i -\gamma_i q^{-1} e_{g\cdot i}) + (\gamma_i e_i -\gamma_i q^{-1} e_{g\cdot i})(g\cdot e_i)\\
&=\gamma_i e_i - \gamma_i q^{-1} e_i e_{g\cdot i} + \gamma_i e_i e_{g\cdot i} - \gamma_i q^{-1} e_{g\cdot i}.
\end{split}
\]
Now the condition \eqref{eq:qgamma} gives that $\gamma_i=0$ whenever $g\cdot i=i$, since $q\neq 1$.  This means that the second and third terms above are equal to 0, and so \eqref{eq:xdoteicheck} is verified in this case.
Now assume $i \neq j$.   Substitution gives
\[
\begin{split}
0&= e_i(\gamma_j e_j -\gamma_j q^{-1} e_{g\cdot j}) + (\gamma_i e_i -\gamma_i q^{-1} e_{g\cdot i})(g\cdot e_j)\\
&= 0 - \gamma_j q^{-1} e_i e_{g\cdot j} + \gamma_i e_i e_{g\cdot j} - 0.
\end{split}
\]
If $i \neq g\cdot j$, the remaining terms are 0, and if  $i =g\cdot j$, the relation \eqref{eq:qgamma} forces the two remaining terms to cancel, so the equation is verified either way.

\medskip

\noindent (B) Taking an arbitrary action of $\uqb$ on $\kk Q_0$, we have an action of $G$ on $Q_0$ by Lemma \ref{lem:gactionvertices}, and it remains to show that $x$ acts by the formula \eqref{eq:xonvertex} for scalars $\gamma_i$ satisfying \eqref{eq:qgamma}.  We have first from $e_i = e_i^2$ that
\begin{equation}\label{eq:xdoteiei}
x\cdot e_i = x\cdot (e_i e_i) = e_i(x\cdot e_i) + (x\cdot e_i)(g\cdot e_i) = (x\cdot e_i)(e_i + e_{g\cdot i}),
\end{equation}
which shows that $x\cdot e_i=\gamma_i e_i + \gamma'_i e_{g\cdot i}$ for some $\gamma_i, \gamma'_i \in \kk$.
For $i$ such that $g\cdot i = i$, it follows directly from \eqref{eq:xdoteiei} that $x\cdot e_i=0$ and thus we may take $\gamma_i=0$ in this case.  Now assume $g\cdot i \neq i$.
Then for $i \neq j$, we have
\begin{equation}\label{eq:xdoteiej}
0= x\cdot (e_j e_i) = e_j(x\cdot e_i) + (x\cdot e_j)(g\cdot e_i) = 0+\gamma'_i e_j e_{g\cdot i} + \gamma_j e_j e_{g\cdot i} + 0.
\end{equation}
Taking $j= g\cdot i$, we find $\gamma_i' = - \gamma_{g\cdot i}$, so that $x\cdot e_i=\gamma_i e_i - \gamma_{g\cdot i} e_{g\cdot i}$.  Finally, the relation $gx=qxg$ in $\uqb$ applied to this forces
\begin{equation}
\gamma_i e_{g\cdot i} - \gamma_{g\cdot i} e_{g^2\cdot i} = \gamma_{g\cdot i} q e_{g\cdot i} - \gamma_{g^2 \cdot i}e_{g^2\cdot i},
\end{equation}
and by comparing coefficients we confirm that $\gamma_i = \gamma_{g\cdot i} q$, so every action comes from data as in part (A).
\end{proof}

\begin{corollary}\label{cor:uqbonvertices1}
If $q$ is not a root of unity, then every Hopf action of $\uqb$ on $\kk Q_0$ factors through $\uqb/\langle g^n-1, x\rangle \simeq \kk(\ZZ/n\ZZ)$ for some $n \in \ZZ^+$.
\end{corollary}
\begin{proof}
Since $Q_0$ is finite, every $G$-orbit is finite, and so \eqref{eq:qgamma} implies that every $\gamma_i=0$ if $q$ is not a root of unity.
\end{proof}

The next corollary (together with Proposition \ref{prop:uqbonvertices}) gives a generalization of  \cite[Prop.~3.5]{KW16} from ordinary Taft algebras to generalized Taft algebras.

\begin{corollary}\label{cor:uqbonvertices2}
Suppose $q$ is a primitive $r^{\rm th}$ root of unity and consider an action of $\uqb$ on $\kk Q_0$.  Then for all $i \in Q_0$, either $r \mid \#(G\cdot i)$ or $\gamma_i=0$.
The action factors through the generalized Taft algebra $T(r,n)$ if and only if, for all $i \in Q_0$, we have $\#(G\cdot i) \mid n$ and either $\#(G\cdot i)= r$ or $\gamma_i = 0$.
\end{corollary}
\begin{proof}
It follows from \eqref{eq:qgamma} that $q^{-\#(G\cdot i)}\gamma_i = \gamma_i$ for all $i \in Q_0$, so either $\gamma_i=0$ or $r$ divides $\#(G\cdot i)$.  
We need a preliminary computation: from Proposition \ref{prop:deltalxk} we have that 
\begin{equation}\label{eq:xdab}
x^r\cdot e_i=x^r \cdot (e_ie_i) = \sum_{\lambda \in WC(r,2)} \sqbinom{r}{\lambda}_{q^{-1}} (x^{\lambda_1}\cdot e_i) (g^{\lambda_1}x^{\lambda_2} \cdot e_i) = (x^r \cdot e_i)(g^r \cdot e_i) + e_i(x^r \cdot e_i),
\end{equation}
where the last equality uses \eqref{eq:binomvanish}. This shows that $x^r\cdot e_i$ is of the form $\alpha e_i + \beta e_{g^r\cdot i}$ for some $\alpha, \beta \in \kk$.
Then we can directly compute $\alpha, \beta$ from iterated application of \eqref{eq:xonvertex}, using property \eqref{eq:qgamma}.  Namely, we have
\begin{equation}\label{eq:xdei}
\begin{split}
x^r\cdot e_i &= \gamma_i^r e_i + (-1)^r \gamma_i \cdots \gamma_{g^{r-1}\cdot i} q^{-r} e_{g^r\cdot i} \\
&= \gamma_i^r e_i + (-1)^r \gamma_i^r q^{-1} \cdots q^{-(r-1)} q^{-r} e_{g^r\cdot i} 
= \gamma_i^r(e_i - e_{g^r\cdot i}),
\end{split}
\end{equation}
where the first equality uses that $r \leq \#(G\cdot i)$ guarantees the $e_{g^k\cdot i}$ are linearly independent for $1<k<r$ whenever $\gamma_i\neq 0$,
The last equality above uses the identity 
\begin{equation}
q^{-1} \cdots q^{-(r-1)} q^{-r}=q^{-\frac{r(r+1)}{2}} =
\begin{cases}
1 & r\ \text{odd}\\
-1 & r\ \text{even},
\end{cases}
\end{equation}
so that the coefficient of $e_{g^r \cdot i}$ is $-\gamma_i^r$ regardless of the parity of $r$.

Now assume the action factors through $T(r,n)$.
Since $g^n=1$ in $T(r,n)$, we see that $\#(G\cdot i)$ divides $n$ for all $i \in Q_0$.
Since $x^r=0$, we get from \eqref{eq:xdei} that $g^r\cdot i = i$ whenever $\gamma_i \neq 0$ and thus $\#(G\cdot i)=r$ since $r$ divides $\#(G\cdot i)$.  

The converse is immediate, again using \eqref{eq:xdei}.
\end{proof}

\subsection{Actions of $\uqb$ and Taft algebras on path algebras}
We now extend the results of the previous section by adding arrows to the quivers.

\begin{theorem}
\label{thm:uqbonkQ}
Let $Q$ be a quiver.  

\noindent (a) The following data determines a Hopf action of $\uqb$ on $\kk Q$.
\begin{enumerate}[(i)]
\item A Hopf action of $\uqb$ on $\kk Q_0$ (see Proposition \ref{prop:uqbonvertices});
\item A representation of $G$ on $\kk Q_1$ satisfying $s(g\cdot a) = g\cdot sa$ and $t(g\cdot a) = g\cdot ta$ for all $a \in Q_1$.
\item A $\kk$-linear endomorphism $\sigma: \kk Q_0 \oplus \kk Q_1 \to \kk Q_0 \oplus \kk Q_1$
\end{enumerate}
satisfying
\begin{enumerate}[($\sigma$1)]
\item $\sigma(\kk Q_0)=0$;
\item $\sigma(a) = e_{sa}\sigma(a)e_{g\cdot ta}$ for all $a \in Q_1$;
\item $\sigma(g\cdot a) = q^{-1} g\cdot \sigma(a)$ for all $a \in Q_1$.  
\end{enumerate}
With this data, the $x$-action is given on $a \in Q_1$ by
\begin{equation} \label{eq:xonarrow}
x \cdot a = \gamma_{ta} a - \gamma_{sa}q^{-1}(g\cdot a) + \sigma(a).
\end{equation}

\noindent (b) Every (filtered) action of $\uqb$ on $\kk Q$ is of the form above.
\end{theorem}
\begin{proof}
(A) It is straightforward to check that the given data makes $\kk Q_0 \oplus \kk Q_1$ into a $\uqb$-module, thus inducing a $\uqb$-module structure on $\kk Q\cong T_{\kk Q_0}(\kk Q_1)$ as in Lemma \ref{lem:gactionvertices}.
Furthermore, assumption (ii) is equivalent to the $G$-action on $\kk Q_1$ in (ii) satisfying the axioms of a Hopf action.
It remains to show that the proposed action of $x$ satisfies the axioms of a Hopf action, by checking that it is compatible with the relations \eqref{eq:kQrelations} defining $\kk Q$.

We first perform a preliminary computation that will simplify the remainder of the proof: namely, given the data of \textit{(i)} and \textit{(ii)}, we have:
\begin{equation}\label{eq:xonarrow2}
(x \cdot e_{sa}) (g\cdot a) + a(x \cdot e_{ta})  = \gamma_{ta} a - \gamma_{sa}q^{-1}(g\cdot a) .
\end{equation}
This follows by direct substitution of \eqref{eq:xonvertex}, recalling that for any $i \in Q_0$, we have $g\cdot i = i$  implies $\gamma_i =0$.

It remains to check that the stated action of $x$ preserves the relations involving an arrow in \eqref{eq:kQrelations}.  For the first, we get
\begin{equation}
\begin{split}
x \cdot (e_{sa} a) &= e_{sa}(x \cdot a) + (x \cdot e_{sa})(g\cdot a)\\
&= e_{sa}(x \cdot e_{sa}) (g\cdot a) + e_{sa}a(x \cdot e_{ta}) + e_{sa}\sigma(a) + (x\cdot e_{sa})(g\cdot a),
\end{split}
\end{equation}
and we wish to show that this is equal to  \eqref{eq:xonarrow}.  Substituting in \eqref{eq:xonarrow2} and simplifying using ($\sigma 2$) shows that the equality of the two expressions is equivalent to the vanishing of $e_{sa}(x\cdot e_{sa})(g\cdot a)$. This can be calculated by substituting \eqref{eq:xonvertex}:
\begin{equation}
e_{sa}(x\cdot e_{sa})(g\cdot a) = \gamma_{sa}e_{sa}(g\cdot a) - \gamma_{sa}q^{-1} e_{sa} e_{g\cdot sa} (g\cdot a).
\end{equation}
If $\gamma_{sa} =0$, then the term clearly vanishes.  If $\gamma_{sa} \neq 0$, then we know $sa \neq g\cdot sa (=s(g\cdot a))$, which also causes the term to vanish.
The other relation $x\cdot a = x \cdot (a e_{ta})$ is checked similarly, and thus the given data determines a Hopf action of $\uqb$ on $\kk Q$.

\medskip

\noindent (B) Taking an arbitrary action of $\uqb$ on $\kk Q$, it restricts to an action on $\kk Q_0$ by the filtration assumption (\ref{assume}) on our Hopf actions, giving \textit{(i)}.  We get \textit{(ii)} from Lemma \ref{lem:gactionvertices}, so it remains to show the existence of a map $\sigma$ as in \textit{(iii)} such that the formula \eqref{eq:xonarrow} holds.
Define this map by  $\sigma(\kk Q_0)=0$, and for $a \in \kk Q_1$ set 
\begin{equation}\label{eq:sigmaformula}
\sigma(a):=e_{sa}(x\cdot a) - a (x\cdot e_{ta})
\end{equation}
It is immediate to see that property $(\sigma3)$ is satisfied, from the relation $xg=q^{-1}gx$, and we will return to $(\sigma2)$ momentarily.

From the relation $a = e_{sa}a$ we get the expression
\begin{equation}\label{eq:xarelation1}
\begin{split}
x\cdot a= x \cdot (e_{sa} a) &= e_{sa}(x \cdot a) + (x \cdot e_{sa})(g\cdot a)\\
&= e_{sa}(x \cdot a) + (x \cdot e_{sa})(g\cdot a) +a(x\cdot e_{ta}) - a(x\cdot e_{ta})\\
&= (x \cdot e_{sa})(g\cdot a) +a(x\cdot e_{ta}) + \sigma(a),
\end{split}
\end{equation}
verifying that \eqref{eq:xonarrow} holds (via \eqref{eq:xonarrow2}).  It remains to show property $(\sigma2)$.  Note that $\sigma(a)= e_{sa}\sigma(a)$ follows directly from the definition \eqref{eq:sigmaformula}.  To get $\sigma(a)=\sigma(a)e_{ta}$, we need another formula for $\sigma$: this comes from  the relation $a = ae_{ta}$, which gives
\begin{equation}\label{eq:xarelation2}
\begin{split}
x\cdot a= x \cdot (ae_{ta} ) &=a(x\cdot e_{ta}) + (x\cdot a)(g\cdot e_{ta}) \\
&= a(x\cdot e_{ta}) + (x\cdot a)(g\cdot e_{ta}) +(x\cdot e_{sa})(g\cdot a) - (x\cdot e_{sa})(g\cdot a)\\
&= (x \cdot e_{sa})(g\cdot a) + a(x\cdot e_{ta}) + (x\cdot a)(g\cdot e_{ta}) - (x\cdot e_{sa})(g\cdot a).
\end{split}
\end{equation}
Comparing this with the last line of \eqref{eq:xarelation1}, we get the expression
\begin{equation}\label{eq:sigmaformula2}
\sigma(a)=(x\cdot a)(g\cdot e_{ta}) - (x\cdot e_{sa})(g\cdot a), 
\end{equation}
which shows that $\sigma(a) = \sigma(a)e_{g\cdot ta}$.  Thus property $(\sigma2)$ holds, and the proof is completed.
\end{proof}

\begin{corollary}\label{cor:uqbonkQ1}
Suppose $q$ is not a root of unity.   For any action of $\uqb$ on $\kk Q$, the map $\sigma$ of Theorem \ref{thm:uqbonkQ} is nilpotent, and so $x\cdot a = \sigma(a)$ for all $a \in Q_1$. 

\end{corollary}
\begin{proof}
From Corollary \ref{cor:uqbonvertices1}, we know that $x\cdot e_i=0$ (i.e. $\gamma_i=0$) for all $i \in Q_0$, so that $x\cdot a = \sigma(a)$ by \eqref{eq:xonarrow}.
To show $\sigma$ is nilpotent, we can pass to an algebraic closure of $\kk$; to simplify the notation, we just assume without loss of generality that $\kk$ is algebraically closed.  For any $\lambda \in \kk$, consider the generalized eigenspace
\begin{equation}
(\kk Q_1)_\lambda=\setst{a \in \kk Q_1}{ (g - \lambda\,1)^M\cdot a  =0 \text{ for some }M\gg 0 },
\end{equation}
noting that repeated application of property $(\sigma3)$ gives $\sigma^n((\kk Q_1)_\lambda) \subseteq (\kk Q_1)_{q^{n} \lambda}$ for $n \geq 0$.
Considering the decomposition into generalized eigenspaces,
$$\kk Q_1 = \bigoplus_{\lambda \in \kk} (\kk Q_1)_\lambda,$$
it enough to show that each $(\kk Q_1)_\lambda$ is annihilated by some power of $\sigma$.
If $\lambda=0$ we are done, and if $\lambda \neq 0$, we get $(\kk Q_1)_{q^{n} \lambda} \cap (\kk Q_1)_{q^{m} \lambda}=0$ for $n \neq m$ since $q$ is not a root of unity.  Thus for fixed $\lambda$, we must have that $(\kk Q_1)_{q^{n} \lambda}=0$ for $n \gg 0$, since $\kk Q_1$ is finite dimensional. This shows that $\sigma$ is nilpotent.
\end{proof}

We note that the case of $r=n$ in the following corollary was first proven in the Ph.D. thesis of Ana Berrizbeitia \cite{Berrizbeitia18}.

\begin{corollary}\label{cor:uqbonkQ2}
Suppose $q$ is a primitive $r^{\rm th}$ root of unity.  In this case, an action of $\uqb$ on $\kk Q$ factors through a generalized Taft algebra $T(r,n)$ if and only if the following hold.
\begin{enumerate}
\item The action on $\kk Q_0$ factors through $T(r,n)$.
\item For all $a \in Q_1$ we have
\begin{equation}\label{eq:uqbTaft}
\gamma_{sa}^r\, g^r\cdot a - \gamma_{ta}^r a = \sigma^r(a).
\end{equation}
\item The element $g^n$ acts as the identity on all of $\kk Q$. 
\end{enumerate}
\end{corollary}
\begin{proof} 
The same reasoning used in \eqref{eq:xdab}, applied to the relation $a=e_{sa}a$, gives 
\begin{equation}
x^r\cdot a = (x^r \cdot e_{sa})(g^r \cdot a) + e_{sa}(x^r \cdot a).
\end{equation}
Note, if the action of $\uqb$ on $\kk Q$ factors through $T(r,n)$, the action of $\kk Q_0$ factors through $T(r,n)$, and therefore Corollary \ref{cor:uqbonvertices2} holds in either direction of the corollary. Thus, we can substitute for $x^r \cdot e_{sa}$ using \eqref{eq:xdei} to get
\[
x^r\cdot a = \gamma_{sa}^r e_{sa}(g^r \cdot a) - \gamma_{sa}^r e_{g^r\cdot sa}(g^r \cdot a) + e_{sa}(x^r \cdot a) = e_{sa}(x^r\cdot a), 
\]
where the second equality uses that either $\gamma_{sa}=0$ or $sa = g^r\cdot sa (=s(g^r \cdot sa))$ to get the first two terms to vanish or cancel.  A similar computation shows that $x^r \cdot a = (x^r \cdot a)e_{ta}$.

Knowing now that $x^r \cdot a = e_{sa} (x^r \cdot a)e_{ta}$, we can explicitly calculate it.
We see that the sum resulting from $r$ applications of the formula \eqref{eq:xonarrow} results in a linear combination of terms of the form $\sigma^i (g^j \cdot a)$ with $i, j \in \ZZ_{\geq 0}$ and $i+j \leq r$ where the scalars contain positive powers of $\gamma_{sa}$ or $\gamma_{ta}$ for all but the $\sigma^r(a)$ term. 

For any term with a positive power of $\gamma_{sa}$ and $j<r$, we have $s(\sigma^i(g^j\cdot a)) = s(g^j \cdot a)$. Then, by Corollary \ref{cor:uqbonvertices2}, either $\gamma_{sa} = 0$ or $s(g^j \cdot a)\neq sa$. For any terms containing only positive powers of $\gamma_{ta}$ and no positive powers of $\gamma_{sa}$, we have that $i+j < r$. Since $t(\sigma^i(g^j \cdot a) = t(g^{i+j}\cdot a)$, Corollary \ref{cor:uqbonvertices2} implies that either $\gamma_{ta} = 0$ or $t(g^{i+j}\cdot a) \neq ta$. 

Since $x^r \cdot a \in e_{sa}\kk Q_1 e_{ta}$, this reduces us to consideration of the pairs $(0,0),\ (r,0),\ (0,r)$.
The coefficients of the corresponding terms are computed using the same method as in \eqref{eq:xdei} to get
\begin{equation}
x^r\cdot a = \gamma_{ta}^r a - \gamma_{sa}^r g^r\cdot a + \sigma^r(a),
\end{equation}
which shows that the action factors through $T(r,n)$ if and only if  \eqref{eq:uqbTaft} holds.
\end{proof}

\subsection{Actions of $\Uqsl$ and $\uqsl$ on path algebras}
We again start with parametrizing actions of $\Uqsl$ and $\uqsl$ on quivers without arrows.
We let $G=\langle K\rangle$.

\begin{proposition}
\label{prop:uqslonvertices}
Let $Q_0$ be the vertex set of a quiver.  

\noindent (A) The following data determines a Hopf action of $\Uqsl$ on $\kk Q_0$.
\begin{enumerate}[(i)]
\item A permutation action of $G$ on the set $Q_0$;
\item Two collection of scalars $(\gamma_i^E \in \kk)_{i \in Q_0}$ and $(\gamma_i^F \in \kk)_{i \in Q_0}$ such that
\begin{equation}\label{eq:qgammaEF}
\gamma^E_{K\cdot i} = q^{-2} \gamma^E_i \quad \text{and} \quad \gamma^F_{K\cdot i} = q^{2} \gamma^F_i \quad \forall i \in Q_0.
\end{equation}
For each $i \in Q_0$ such that $K^2 \cdot i \neq i$, these scalars must furthermore satisfy
\begin{equation}\label{eq:gammaEFcondition}
\gamma^E_i\gamma^F_i =\frac{-q}{(1-q^2)^2}.
\end{equation}
\end{enumerate}
The action is given by
\begin{equation} \label{eq:EFonvertex}
\begin{split}
E \cdot e_i &= \gamma^E_i e_i - \gamma^E_i q^{-2} e_{K\cdot i} \quad {\text{ for all $i \in Q_0$.}}\\
F \cdot e_i &= \gamma^F_i e_{K^{-1}\cdot i} - \gamma^F_i q^{2} e_{i} \quad {\text{ for all $i \in Q_0$.}}
\end{split}
\end{equation}

\noindent (B) Every action of $\Uqsl$ on $\kk Q_0$ is of the form above.
\end{proposition}
\begin{proof}
(A)
The isomorphisms \eqref{eq:borels} along with Proposition \ref{prop:uqbonvertices} show that the given data define actions of the subalgebras $\langle K,E\rangle$ and $\langle K, F\rangle$ on $\kk Q_0$.
These subalgebras together generate $\Uqsl$,
and the only compatibility condition for such a pair of actions to define an action of $\Uqsl$ is that the rightmost relation of \eqref{eq:uqslrelations} holds.
It can be directly computed using Proposition \ref{prop:uqbonvertices} that
\begin{equation}
(EF -FE)\cdot e_i = \gamma_i^E\gamma_i^F(1-q^2)(e_{K\cdot i}-e_{K^{-1}\cdot i}),
\end{equation}
which must be equal to $(q-q^{-1})^{-1}(e_{K\cdot i} - e_{K^{-1}\cdot i})$.  
If $K^2\cdot i=i$ then this condition is vacuous.  If not, then equality of these two expressions is equivalent the condition \eqref{eq:gammaEFcondition}.

\noindent (B)  This is immediate from Proposition \ref{prop:uqbonvertices}(B) along with the observations above.
\end{proof}

\begin{remark}
We note that for actions in the above Proposition where $K^2$ does not fix any vertices, equation \eqref{eq:gammaEFcondition} implies that
each of $\gamma_i^E, \gamma_i^F$ is nonzero and determines the other, so the action is completely determined by its restriction to a quantum Borel, an analogue of a result in \cite{MS}.  It may also be interesting to compare the present work to actions of $\Uqsl$ and related algebras constructed in \cite{MS90}.
\end{remark}

\begin{corollary}\label{cor:uqslonvertices1}
If $q$ is not a root of unity, then every Hopf action of $\Uqsl$ on $\kk Q_0$ factors through $\Uqsl/\langle K^2-1,E,F\rangle \simeq \kk(\ZZ/2\ZZ)$.
\end{corollary}
\begin{proof}
From the isomorphisms \eqref{eq:borels} and Corollary \ref{cor:uqbonvertices1}, we see that $\gamma_i^E=\gamma_i^F=0$ for all $i \in Q_0$, so the action factors through $\Uqsl/\langle K^n-1,E,F\rangle$ for some $n \in \ZZ^+$.  But then \eqref{eq:gammaEFcondition} can never be satisfied since the right hand side is nonzero, so $K^2\cdot i =i$ for all $i \in Q_0$.
\end{proof}

\begin{corollary}\label{cor:uqslonvertices2}
Let $n=|q|$ and suppose $n>2$ and odd, and we have an action of $\Uqsl$ on $\kk Q_0$.
Then for all $i \in Q_0$, we have either $\#(G\cdot i) = 1, 2$, or is divisible by $n$. Such an action factors through $\uqsl$ if and only if every $G$-orbit on $Q_0$ has 1 or n vertices.

\end{corollary}
\begin{proof}
Since $n$ is odd, $|q^2|=n$ as well.
Applying Corollary \ref{cor:uqbonvertices2} to the restricted actions of $\langle E, K \rangle \simeq U_{q^2}(\fb)$ and $\langle F, K \rangle \simeq U_{q^{-2}}(\fb)$ from \eqref{eq:borels}, 
we see that for each $i \in Q_0$, either $n$ divides $\#(G\cdot i)$ or $\gamma_i ^E=\gamma_i^F = 0$.  
If $\gamma_i ^E=\gamma_i^F = 0$, then \eqref{eq:gammaEFcondition} can never be satisfied since the right hand side is nonzero, so we have $K^2\cdot i =i$ in this case.  This proves the first statement.

Now the $\Uqsl$ action factors through $\uqsl$ if and only if the restricted actions of the subalgebras $\langle K, E\rangle$ and $\langle K, F\rangle$ factor through $\langle K, E\rangle/\langle K^n-1, E^n\rangle$ and $\langle K, F\rangle/\langle K^n-1, F^n\rangle$, which are each isomorphic to the Taft algebra $T(n)$.  Corollary \ref{cor:uqbonvertices2} tells us that this happens if and only if $\#(G\cdot i) \mid n$ and $\#(G\cdot i) =n$ when $\gamma_i^E\neq 0$ or $\gamma_i^F\neq 0$.  Using the first part of this corollary, $\#(G\cdot i) \mid n$ is equivalent to $\#(G\cdot i)=1$ or $n$ (since $n$ is odd). For the second condition, if $\gamma_i^E\neq 0$ or $\gamma_i^F\neq 0$ then $\#(G\cdot i)\geq n > 2$ again by Corollary \ref{cor:uqbonvertices2}, so this $\#(G\cdot i)=n$ by the first part of this corollary, and thus this condition is automatically satisfied.
\end{proof}

We now proceed to parametrize actions of $\Uqsl$ and $\uqsl$ on arbitrary quivers.

\begin{theorem}
\label{thm:UqslonkQ}
Let $Q$ be a quiver.
\noindent (A) The following data determines a Hopf action of $\Uqsl$ on $\kk Q$.
\begin{enumerate}[(i)]
\item A Hopf action of $\Uqsl$ on $\kk Q_0$ (as in Proposition \ref{prop:uqslonvertices});
\item a representation of $G$ on $\kk Q_1$ satisfying $s(K\cdot a) = K\cdot sa$ and $t(K\cdot a) = K\cdot ta$ for all $a \in Q_1$;
\item a pair of $\kk$-linear endomorphisms $\sigma^E, \sigma^F: \kk Q_0 \oplus \kk Q_1 \to \kk Q_0 \oplus \kk Q_1$
satisfying
\begin{enumerate}[($\sigma$1)]
\item $\sigma^\bullet(\kk Q_0)=0$ for $\bullet \in \{E, F\}$;
\item $\sigma^\bullet(a) = e_{sa}\sigma^\bullet(a)e_{K\cdot ta}$ for $\bullet \in \{E, F\}$ and all $a \in Q_1$;
\item $\sigma^E(K\cdot a) = q^{-2} K\cdot \sigma^E(a)$ and $\sigma^F(K\cdot a) = q^{2} K\cdot \sigma^F(a)$ for all $a \in Q_1$;
\item $\gamma^E_{sa}\gamma^F_{sa}(1-q^2)K^2\cdot a -\gamma^E_{ta}\gamma^F_{ta}(1-q^2)a+q^2\sigma^E(\sigma^F(a)) - \sigma^F(\sigma^E(a))\\ = (q-q^{-1})^{-1}(K^2\cdot a - a)$ for all $a \in Q_1$.
\end{enumerate}
\end{enumerate}
The action is given on $a \in Q_1$ by
\begin{equation} \label{eq:EFonarrow}
\begin{split}
E \cdot a &= \gamma^E_{ta} a - \gamma^E_{sa}q^{-2}(K\cdot a) + \sigma^E(a)\\
F \cdot a &= \gamma^F_{ta} (K^{-1}\cdot a) - \gamma^F_{sa}q^{2}a + K^{-1}\cdot \sigma^F(a)\\
\end{split}
\end{equation}

\noindent (B) Furthermore, every (filtered) action of $\Uqsl$ on $\kk Q$ with $Q$ connected is of the form above.
\end{theorem}
\begin{proof}
The isomorphisms \eqref{eq:borels} along with Theorem \ref{thm:uqbonkQ}(A) show that the data given in (i), (ii), (iii) but omitting ($\sigma 4$) define actions of the subalgebras $\langle K,E\rangle$ and $\langle K, F\rangle$ via \eqref{eq:EFonarrow} on $\kk Q$, and the restrictions of these actions to $\langle K \rangle$ are the same.
These subalgebras together generate $\Uqsl$, and the only compatibility condition for such a pair of actions to define an action of $\Uqsl$ is that the rightmost relation of \eqref{eq:uqslrelations} holds for the action.
It can be directly computed using Theorem \ref{thm:uqbonkQ} that
\begin{equation}\label{eq:EFFE}
\begin{split}
(EF -FE)\cdot a &= \gamma^E_{sa}\gamma^F_{sa}(1-q^2)K\cdot a -\gamma^E_{ta}\gamma^F_{ta}(1-q^2)K^{-1}\cdot a\\
&+q^2K^{-1}\cdot \sigma^E(\sigma^F(a)) - K^{-1}\cdot \sigma^F (\sigma^E(a)),
\end{split}
\end{equation}
so the given data defines an action of $\Uqsl$ on $\kk Q$ if and only if this is equal to
\begin{equation}
(q-q^{-1})^{-1}(K\cdot a - K^{-1}\cdot a)
\end{equation}
for all $a \in Q_1$. Acting by $K$ on both sides of this equality gives us condition ($\sigma 4$).

\noindent (B)  Every action of $\Uqsl$ on $\kk Q$ restricts to actions of the subalgebras $\langle K,E\rangle$ and $\langle K, F\rangle$, which are of the form stated in (A) by the isomorphisms \eqref{eq:borels} and Theorem \ref{thm:uqbonkQ}(B).
\end{proof}

\begin{remark}\label{rem:orbitsgt2}
If $a \in Q_1$ satisfies $\#(G\cdot sa), \#(G\cdot ta) > 2$, then $(\sigma 4)$ in Theorem \ref{thm:UqslonkQ} simplifies to 
\begin{equation}\label{eq:sigmacommutativity}
q^2\sigma^E\sigma^F(a) = \sigma^F\sigma^E(a).
\end{equation}
This is a direct result of substituting \eqref{eq:gammaEFcondition} in to $(\sigma4)$ and simplifying.
\end{remark}

\begin{corollary}\label{cor:uqslonkQ1}
Suppose $q$ is not a root of unity.  For any action of $\Uqsl$ on $\kk Q$, both the maps $\sigma^E$ and $\sigma^F$ are nilpotent, and $E\cdot a = \sigma^E(a)$ and $F\cdot a = K^{-1}\cdot \sigma^F(a)$ for all $a \in Q_1$.  Furthermore, these maps satisfy
\begin{equation}
q^2 \sigma^E(\sigma^F(a)) - \sigma^F(\sigma^E(a)) = \frac{K^2 \cdot a - a}{q-q^{-1}} \quad \text{for all }a \in Q_1.
\end{equation}
\end{corollary}
\begin{proof}
This is immediate from Corollary \ref{cor:uqbonkQ1} and Theorem \ref{thm:UqslonkQ}.
\end{proof}

\begin{corollary}\label{cor:uqslonkQ2}
If $q$ is a primitive $n^{\text{th}}$ root of unity with $n$ odd and $n>2$, the action of $\Uqsl$ on $\kk Q$ factors through $\uqsl$ if and only if the following conditions hold.
\begin{enumerate}
\item $K^n$ acts as the identity on all of $\kk Q$.
\item The action of $\Uqsl$ on $\kk Q_0$ factors through $\uqsl$ (see Corollary \ref{cor:uqslonvertices2}).
\item For any $a \in Q_1$, we have 
\begin{equation}\label{eq:uqslonkQ}
((\gamma_{sa}^{\bullet})^n - (\gamma_{ta}^{\bullet})^n)a = (\sigma^{\bullet})^n(a),
\end{equation}
for $\bullet \in \{E, F\}$. 
\end{enumerate}
\end{corollary}
\begin{proof}
The action of $\Uqsl$ factors through $\uqsl$ if and only if the actions of the Borel subalgebras $\langle E, K \rangle$ and $\langle F, K \rangle$ factor through $T(n)$. The statement follows from Corollary \ref{cor:uqbonkQ2}.
\end{proof}

\section{Tensor categorical viewpoint}\label{sec:bimodules}
We now turn to the language of tensor categories, and in the case of Taft algebras, connect our results with work of Etingof and Ostrik on exact algebras in the category $\rep(T(n))$.

\subsection{Hopf actions and tensor categories}\label{sec:tensorcat}
The study of Hopf actions on quiver path algebras falls within the more general framework of tensor algebras in tensor categories, as in \cite{EKW20}.
Here, the relevant tensor categories are $\mathcal{C}=\rep(H)$ where $H$ is one of the Hopf algebras of \eqref{eq:algebras}.
These categories are not all \emph{finite} in the sense of \cite[\S2.1]{EO} (i.e. equivalent to the category of finite-dimensional representations of some finite-dimensional algebra), but this doesn't cause any problem in the present work, since we do not use results where the finiteness condition is relevant.

Let $H$ be a Hopf algebra, $Q$ a quiver, and let $S=\kk Q_0$ and $V=\kk Q_1$, considered as an $S$-bimodule, so that $\kk Q \cong T_S(V)$.
We say an $H$-action on $T_S(V)$ is \emph{graded} if it preserves the path length grading, which is equivalent to each of $S$ and $V$ being $H$-stable.  A graded $H$-action makes $S$ into an algebra in $\mathcal{C}$ and $V$ an $S$-bimodule in $\mathcal{C}$. 
In the language of \cite{EKW20}, we say $T_S(V)$ is a $\mathcal{C}$-tensor algebra.  
It is furthermore a \emph{minimal, faithful} $\mathcal{C}$-tensor algebra if $V$ is an indecomposable $S$-bimodule in $\mathcal{C}$, and no two-sided ideal of $S$ in $\mathcal{C}$ acts by 0 on $V$.  Intuitively, these are the buildings blocks of all $\mathcal{C}$-tensor algebras.

To study minimal, faithful tensor algebras in $\mathcal{C}$ we may assume $S$ has only 1 or 2 indecomposable summands as an algebra in $\mathcal{C}$ (see \cite[Rmk~3.15]{EKW20}).
We do not consider the notion of equivalence of tensor algebras in this paper, so the case of $S$ being indecomposable is subsumed by the case of $S$ having 2 indecomposable summands, by taking $S_1=S_2$ in the study of $S_1$-$S_2$-bimodules in $\mathcal{C}$.

In the following sections, we consider categories of $S_1$-$S_2$-bimodules in $\mathcal{C}$, where $S_1$, $S_2$ are commutative algebras in $\mathcal{C}$.
We establish an equivalence between any such category of bimodules, 
and a subcategory of certain finite-dimensional representations of an associative $\kk$-algebra, explicitly given in terms of quivers with relations.

\subsection{The relevant quivers and bimodules}\label{sec:quiversbimods}
Fix a positive integer $N$ and $\mu \in \kk^{\times}$. To this data we associate an infinite quiver $\cQ(\mu,N)$ whose vertex set is
$\kk^\times \times \ZZ/N\ZZ$.
Coming out of each vertex $(\lambda,k)$ there is a loop, and an arrow $(\mu\lambda,k+1) \to (\lambda,k)$. We refer to the loops as \emph{$b$-type arrows} and the arrows $(\mu\lambda,k+1)\to (\lambda,k)$ as \emph{$a$-type arrows}. When it is not clear from context, we index these by their target vertex.
The following lemma (whose easy proof is omitted) gives the form of the connected components of this quiver.

\begin{lemma}\label{lem:connectedcomps}
If $\mu$ is not a root of unity, then the connected components of $\cQ(\mu,N)$ all have the form $\cS_\infty$, where $\cS_\infty$ is the infinite quiver shown below.
\begin{equation}\label{eq:Sinfinity}
\cS_\infty=
\vcenter{\hbox{
\begin{tikzpicture}[point/.style={shape=circle,fill=black,scale=.5pt,outer sep=3pt},>=latex]
\node (dots1) at (0,0) {$\cdots$};
\node[point,label={below:$(\mu\lambda,k+1)$}] (-1) at (2,0) {} edge[in=135,out=45,loop] node[above] {$b$} ();
\node[point,label={below:$(\lambda,k)$}] (0) at (5,0) {} edge[in=135,out=45,loop] node[above] {$b$} ();
\node[point,label={below:$(\mu^{-1}\lambda,k-1)$}] (1) at (8,0) {} edge[in=135,out=45,loop] node[above] {$b$} ();
\node (dots2) at (11,0) {$\cdots$};
\path[->] (dots1) edge node[above] {} (-1); 
\path[->] (-1) edge node[above] {$a$} (0);
\path[->] (0) edge node[above] {$a$} (1) ;
\path[->] (1) edge node[above] {$a$} (dots2) ;
\end{tikzpicture} }}
\end{equation}

If $\mu$ is a root of unity, each connected component of $\cQ(\mu,N)$ has the form shown below, where $\cS_p$ has $p:=\lcm(|\mu|,N)$ vertices.
\begin{equation}\label{eq:Sr}
\cS_p=
\vcenter{\hbox{
\begin{tikzpicture}[point/.style={shape=circle,fill=black,scale=.5pt,outer sep=3pt},>=latex]
\node[point,label={-100:$(\mu^{p-1}\lambda,p-1)$}] (0) at (0,0) {} edge[in=135,out=45,loop] node[above] {$b$} ();
\node[point,label={below:$(\mu^{p-2}\lambda,p-2)$}] (1) at (3,0) {} edge[in=135,out=45,loop] node[above] {$b$} ();
\node (dots) at (6,0) {$\cdots$};
\node[point,label={-75:$(\lambda,0)$}] (r) at (9,0) {} edge[in=135,out=45,loop] node[above] {$b$} ();
\path[->] (0) edge node[above] {$a$} (1); 
\path[->] (1) edge node[above] {$a$} (dots);
\path[->] (dots) edge node[above] {$a$} (r) ;
\path[->] (r) edge [bend left=30] node[midway, above] {$a$} (0);
\end{tikzpicture} }}
\end{equation}
\end{lemma}

Let $\Gamma(\mu,N)$ be the quotient of $\kk\cQ(\mu,N)$ by the ideal generated by all relations of the form, recalling our convention from Section \ref{sec:quivers} of reading paths from left to right.
\begin{equation}\label{eq:Qrelations}
ba = \mu ab.
\end{equation} 
Let $\nrep{\Gamma(\mu,N)}$ be the category of finite-dimensional representations of $\Gamma(\mu,N)$ for which the linear map assigned to each loop is nilpotent (but larger oriented cycles are not necessarily nilpotent).
We typically denote a representation of $\Gamma(\mu,N)$ by the shorthand $W=(W_{\lambda,k},\ A_{\lambda,k},\ B_{\lambda,k})$, where $W_{\lambda,k}$ is the vector space associated to vertex $(\lambda,k)$, and (recalling our contravariant convention for representations from Section \ref{sec:quivers})
\begin{equation}
    A_{\lambda,k} \in \Hom_\kk(W_{\lambda,k}, W_{\mu\lambda,k+1}), \qquad B_{\lambda,k} \in \End_\kk(W_{\lambda,k})
\end{equation}
are the maps associated to the arrows $a_{\lambda,k}$ and $b_{\lambda,k}$ respectively.

For a positive integer $m$, consider the algebra $S_m$ with action of $G=\langle g \rangle \simeq \ZZ$ defined as follows. We define $S_m=\kk^m$ as a vector space, with coordinate-wise multiplication.  Letting $e_i$ for $0 \leq i \leq m-1$ be the $i^{\text{th}}$ standard basis vector in $S_m$, we get a complete system of primitive orthogonal idempotents $\{e_0, \dotsc, e_{m-1}\}$.  
The action of $G$ is by $g \cdot e_i = e_{i+1}$, where we always interpret the subscript modulo $m$.  We identify $S_m$ with the path algebra of a quiver with $m$ vertices and no arrows.

\begin{notation}\label{not:uqb}
To study minimal and faithful $\rep(\uqb)$-tensor algebras, we fix the following notation. 
\begin{itemize}
    \item a pair of positive integers $(m,m')$
    \item $\ell=\lcm(m,m')$ and $d =\gcd(m,m')$
    \item $\cQ:=\cQ(q^\ell,d)$ as defined at the beginning of this subsection
    \item $\Gamma(q^\ell,d)$ is the quotient of $\kk \cQ(q^\ell, d)$ by the ideal generated by relations of the form \eqref{eq:Qrelations}
    \item $\cN:=\nrep{\Gamma(q^\ell,d)}$ as defined just above \eqref{eq:Qrelations}
    \item $S_m,\ S_{m'}$ are the algebras defined immediately above -- through the end of  Section \ref{sec:EO}, we assume these algebras have fixed $\uqb$-actions extending the $G$-action
    \item $\cB$ is the category of (finite-dimensional) $S_m$-$S_{m'}$-bimodules in $\rep(\uqb)$
\end{itemize}
\end{notation}

In the following subsections we construct mutually quasi-inverse equivalences 
\begin{equation}
    \Phi\colon \cB \to \cN, \qquad \Psi \colon \cN \to \cB.
\end{equation}

We need a little more technical notation for construction of these functors.
Let $R\subset \kk^\times$ be a set of coset representatives for the subgroup $\langle q^\ell \rangle$ in $\kk^\times$,
and let $(\lambda, k)$ be a vertex of $\cQ$.
Define a function
\begin{equation}\label{eq:tau}
    \tau\colon \cQ_0 \to \ZZ
\end{equation}
in two cases, depending whether $q$ is a root of unity or not.  
Given $(\lambda, k) \in \cQ_0$, let $\lambda_0 \in R$ be the coset representative of $\lambda$.
If $q$ is not a root of unity, there exists a unique $s \in \ZZ$ such that $\lambda = q^{\ell s}\lambda_0$.  In this case, we take $\tau(\lambda,k) \in \ZZ$ minimal such that $\tau(\lambda,k) \geq s$ and $\tau(\lambda,k) = k \mod d$.
If $q$ is a root of unity, there exists a unique $0 \leq \tau(\lambda,k) < \lcm(|q^\ell|,d)$ such that 
both $\lambda=q^{\ell \tau(\lambda,k)}\lambda_0$ and $\tau(\lambda,k) =  k \mod d$.

When $q$ is a root of unity and, we also define a ``correction factor'' associated to the rightmost vertex in \eqref{eq:Sr}.  This is only needed for technical purposes later:
\begin{equation}\label{eq:epsilon}
\ze=\ze(q, m,m',\zl,j):= \begin{cases}
z_1 m & \text{if $q$ is a root of unity and } \tau(\lambda,k) = \lcm(|q^\ell|,d)-1\\
0 & \text{otherwise}.
\end{cases}
\end{equation}
Here, $z_1, z_2 \in \ZZ$ are chosen so that
$\lcm(|q^\ell|,d)= z_1m + z_2m'$ (which is possible since the left hand side is a multiple of $d$).

\subsection{Unraveling bimodules to quiver representations} 
In this section, we construct a functor $\Phi\colon \cB \to \cN$ which unravels a bimodule to extract a quiver representation that minimally encodes the data defining the bimodule.  We will see that this functor admits a quasi-inverse in the next section.

Retaining the notation above, we fix a bimodule $V \in \cB$.
We wish to construct a representation $W \in \cN$ using the data from Theorem \ref{thm:uqbonkQ} that determines the action of $\uqb$ on $V$.  
Recall that $V$ can be decomposed into arrow spaces
\begin{equation}\label{eq:eigenspacedecomp}
V = \bigoplus_{i,j} V_j^i, \quad \text{ where } V^i_j := e_i V e_j,
\end{equation}
where the superscripts are interpreted modulo $m$ and the subscripts modulo $m'$. Each arrow space $V^i_j$ is a representation of the subgroup $\langle g^\ell \rangle \leq G$.
Decomposing as representations of $\langle g^\ell \rangle$, we get generalized eigenspace decompositions
\begin{equation}\label{eq:eigenspacedecomp2}
    V^i_j = \bigoplus_{\lambda \in \kk^\times}V^i_j(\lambda), \quad V^i_j(\lambda) := \setst{v \in V^i_j}{(g^\ell - \lambda\,1)^M\cdot v = 0 \text{ for } M \gg 0}.
\end{equation}
Recalling the definition of $\tau$ below  \eqref{eq:tau}, we then set
\begin{equation}\label{eq:Wlambda}
W_{\lambda,k}:=V^0_{\tau(\lambda,k)}(\lambda).
\end{equation}

Appying Theorem \ref{thm:uqbonkQ} to the path algebra with $\uqb$-action $\kk Q \simeq T_{S_m\oplus S_{m'}}(V)$,
we get a linear map $\sigma\colon V \to V$ satisfying relation ($\sigma$3), which gives
\begin{equation}\label{eq:sigmamu}
    \sigma(V^i_j(\lambda))\subseteq V^i_{j+1}(q^\ell\lambda) 
    \quad\text{for all }\lambda\in \kk^\times.
\end{equation}
In the case $q$ is not a root of unity, one checks from the definition that $\tau(q^\ell\lambda,k+1)=\tau(\lambda, k)+1$ for any $(\lambda, k)\in \cQ$, thus we have a natural linear map
\begin{equation}\label{eq:Alk}
A_{\lambda,k}:=\sigma |_{W_{\lambda,k}} \colon W_{\lambda,k} \to W_{q^\ell\lambda,k+1}
\end{equation}
for each $(\lambda, k)$.

When $q$ is a root of unity, however, we only have $\tau(q^\ell\lambda,k+1)=\tau(\lambda, k)+1$ when $\tau(\lambda,k) \neq \lcm(|q^\ell|,d)-1=:p-1$, 
so we only get maps of the form \eqref{eq:Alk} over each rightward pointing $a$-type arrow in components of $\cQ$ of the form \eqref{eq:Sr}.
This is because the restriction of $\sigma$ sends 
\begin{equation}\label{eq:Alknew1}
W_{q^{\ell(p-1)}\lambda, p-1}= V^0_{p-1}(q^{\ell(p-1)}\lambda)=V^0_{p-1}(q^{-\ell}\lambda) \longrightarrow V^0_p(\lambda),
\end{equation}
with the target being $G$-equivariantly isomorphic (but not equal) to the intended target of $V^0_0(\lambda)=W_{\lambda,0}$.
For this reason we introduced the ``correction factor'' 
$\ze$ in \eqref{eq:epsilon}: since $\ze\equiv 0 \mod m$ and $\ze \equiv p \mod m'$,
the action of $g^{-\ze}$ gives a functorial isomorphism
\begin{equation}\label{eq:Alknew2}
V^0_p(\lambda) \xto{\sim} V^0_0(\lambda).
\end{equation}
Thus the composition of \eqref{eq:Alknew1} and \eqref{eq:Alknew2} gives a linear map
\begin{equation}\label{eq:Alk2}
    A_{q^{\ell(p-1)}}\lambda,p-1\colon 
    W_{q^{\ell(p-1)}}\lambda,p-1 \to W_{\lambda, 0}
\end{equation}
associated to the lower, leftward pointing arrow in a component of $\cQ$ of the form \eqref{eq:Sr}.

We also have a nilpotent endomorphism $B_{\lambda,k} = \left.(g^\ell - \lambda\,1)\right|_{W_{\lambda,k}}$ of each $W_{\lambda,k}$.
With this, we define a representation of $\cQ$ by 
$\Phi(V)=(W_{\lambda,k},\ A_{\lambda,k},\ B_{\lambda,k})$.
The relation ($\sigma 3$) applied $\ell$ times implies that $B_{q^\ell\lambda,k} A_{\lambda,k}=q^\ell A_{\lambda,k} B_{\lambda,k}$
for every $(\lambda,k) \in \cQ$, 
so $\Phi(V)$ satisfies the relations \eqref{eq:Qrelations}, and thus $\Phi(V) \in \cN$.
This gives the definition of $\Phi$ on objects.

Let $\varphi\colon V \to V'$ be a morphism in $\cB$.
Define $\Phi$ on morphisms by $\Phi(\varphi) = \bar{\varphi}$ where $\bar{\varphi}(\lambda,k) = \left. \varphi \right|_{W_{\lambda,k}}$. 

\begin{proposition}\label{prop:Phi}
The definition above makes $\Phi$ a functor $\cB \to \cN$.
\end{proposition}

\begin{proof} 
Retaining the notation above, write $\Phi(V') = W'=(W'_{\lambda,k},\ A'_{\lambda,k},\ B'_{\lambda,k})$. 
Let $\sigma$ and $\sigma'$ be the maps which determine the action of $x$ on $V$ and $V'$ respectively, as in Theorem \ref{thm:uqbonkQ}.

We must show that $\bar{\varphi}: W \to W'$ is a morphism in $\cN$. Since $\varphi$ is a morphism of bimodules, it preserves arrow spaces. 
Since $\varphi$ is a morphism of $\uqb$-representations, it commutes with the action of $g^\ell$ and thus preserves the generalized eigenspaces of this action. 
It follows that $\left. \varphi \right|_{W_{\lambda,k}}$ is a map $W_{\lambda,k} \to W'_{\lambda,k}$ and that $B'_{\lambda,k}\bar{\varphi}(\lambda,k) = \bar{\varphi}(\lambda,k)B_{\lambda,k}$ for every $(\lambda,k) \in \cQ$.

For $a \in W_{\lambda,k}$, we have $x \cdot \varphi(a) = \varphi(x\cdot a)$, which expands via \eqref{eq:xonarrow} to
\begin{equation}
\gamma_{ta}\varphi(a) -\gamma_{sa}q^{-1}g \cdot \varphi(a) + \sigma'\varphi(a) = \gamma_{ta} \varphi(a) -\gamma_{sa}q^{-1}g\cdot\varphi(a)  + \varphi\sigma(a),
\end{equation}
so $\sigma'\varphi(a) = \varphi\sigma(a)$. 
Therefore, $A'_{\lambda,k} \bar{\varphi}(\lambda,k) = 
\bar{\varphi}(q^\ell\lambda,k+1)A_{\lambda,k}$ and $\bar{\varphi}$ is a morphism in $\cN$. 
\end{proof}

\subsection{Equivalence of $\cB$ and $\cN$}
In this section, we construct a functor $\Psi: \cN \to \cB$ which reverses the process carried out by $\Phi$, taking a quiver representation and constructing a bimodule from it in a minimal way.  More precisely, we will show that $\Psi$ is quasi-inverse to $\Phi$.

Fix $W=(W_{\lambda,k},\ A_{\lambda,k},\ B_{\lambda,k}) \in \cN$.
We consider each space $W_{\lambda,k}$ as a representation of $\langle g^\ell\rangle$ via the action
\begin{equation}\label{eq:glonW}
g^\ell \cdot w = \lambda w + B_{\lambda,k}(w), \qquad \text{for } w \in W_{\lambda,k}.
\end{equation}

We define vector spaces for $0 \leq j < m'$ (recalling $\tau\colon \cQ\to\ZZ$ from the previous section) by
\begin{equation}\label{eq:Wbimod}
\tilde{W}^0_j = 
\bigoplus_{\substack{(\lambda,k) \in \cQ\\ \tau(\lambda,k)=j}}
W_{\lambda,k}.
\end{equation}
Then we define an $S_m$-$S_{m'}$-bimodule structure on
\begin{equation}\label{eq:Wtilde}
\tilde{W} = \bigoplus_{j=0}^{m'-1} \tilde{W}^0_j 
\end{equation}
by $e_0 \tilde{W} e_j = \tilde{W}^0_j$.
Let $H \leq \kk G$ be the subalgebra generated by $g^\ell$. 
Since $H$ acts on $\tilde{W}$, the space $\kk G \otimes_H \tilde{W}$ has an induced action of $\kk G$. 
Moreover, we can extend the $S_m$-$S_{m'}$-bimodule structure on $\tilde{W}$ to $\kk G \otimes_H \tilde{W}$ by 
$e_i(g^t \otimes w)e_j = g^t\otimes (e_{i-t}w e_{j-td})$, which is well defined because $\ell=\lcm(m,m')$.

We now extend the action of $G$ to an action of $\uqb$ by specifying an action of $x$.
Let $\{\gamma_{i}\}_{i=0}^{m-1}$ and $\{\gamma'_{j}\}_{j =0}^{m'-1}$ be the scalars from Proposition \ref{prop:uqbonvertices} determining the actions of $x$ on $S_m$ and $S_{m'}$, respectively. 
For $w \in W_{\lambda,k}$ and $t \in \ZZ$, let $j:=\tau(\lambda,k)$ and define
\begin{equation}\label{eq:xgsw}
x \cdot (g^t\otimes w) = q^{-t}(\gamma'_j g^t\otimes w -\gamma_{0}q^{-1}g^{t+1}\otimes w + g^{t+\ze}\otimes A_{\lambda,k}(w))
\end{equation}
where $j$ is interpreted modulo $m'$ as usual, and $\ze$ is the ``correction factor'' from \eqref{eq:epsilon}.

To see that this action is well-defined on $\kk G \otimes_H \tilde{W}$,
one can check that 
\begin{equation}\label{eq:welldefcheck}
    x\cdot(g^\ell \otimes w) - x\cdot(1 \otimes g^\ell \cdot w)= \gamma'_j(q^{-\ell}-1)g^\ell\otimes w - \gamma_0q^{-1}(q^{-\ell}-1)g^{\ell+1}\otimes w,
\end{equation}
where the relation $q^{\ell} A_{\lambda,k} B_{\lambda,k} = B_{q^\ell\lambda,k+1}A_{\lambda,k}$ needs to be used in the computation.
The first summand on the right hand side always vanishes because either $|q|$ divides $m'$, which divides $\ell$, or $\gamma'_j=0$ from Corollaries \ref{cor:uqbonvertices1} and \ref{cor:uqbonvertices2}.  
The second summand also vanishes by the same reasoning, so the expression \eqref{eq:xgsw} is well-defined.

It follows that $\Psi(W) :=\kk G \otimes_H \tilde{W}$ is an object of $\cB$, giving the definition of $\Psi$ on objects.
For a morphism $\vartheta: W \to W'$ in $\cN$, we can simply take $\Psi(\vartheta) = id_{\kk G}\otimes \vartheta$. It is straightforward to see then that $\Psi$ is a functor.

\begin{proposition}\label{prop:Psi}
We have that $\Psi$ is a functor $\cN \to \cB$. 
\end{proposition}

Finally, we come to the main result of this section.
The key idea of the proof is that the  $\tau$-function was carefully chosen to accomplish the unraveling, extraction, and reconstruction described at the start of the previous two subsections.

\begin{theorem}\label{thm:uqbequivalence}
The functors $\Phi$ and $\Psi$ are mutually quasi-inverse, thus the categories $\cB$ and $\cN$ are equivalent.
\end{theorem}

\begin{proof}
For $V \in \cB$, we have $\Psi(\Phi(V)) = \kk G \otimes_H \tilde{\Phi(V)}$. Consider the map $\xi_V \colon \kk G \otimes_H \tilde{\Phi(V)} \to V$ defined by 
$\xi_V(g^t \otimes w)= g^t \cdot w$. We claim this is an isomorphism in $\cB$.  It is straightforward to check that this is a morphism in $\cB$, so we need to verify it is an isomorphism of vector spaces.

To see that $\xi_V$ is surjective, recall the decomposition of $V$ from \eqref{eq:eigenspacedecomp} and \eqref{eq:eigenspacedecomp2}. It is enough to show that each space of the form $V^0_j(\lambda)$ is in the image, since the action of $G$ can then be used to obtain that all $V^i_j(\lambda)$ are in the image.
If $q$ is not a root of unity, let $k \in \ZZ$ be minimal such that $k \equiv j$ mod $d$ and $k \geq s$ where $s$ is the unique integer so that $\lambda = q^{\ell s}\lambda_0$ for $\lambda_0 \in R$. 
If $q$ is a root of unity, choose $k$ to be the unique integer so that $k \equiv j$ mod $d$ and $0 \leq k < \lcm{(\left|q^\ell\right|,d)}$, and $\lambda=q^{\ell k}\lambda_0$.
In either case this means that $\tau(\lambda,k) =k$, and there exist integers $y_1,y_2$ with 
\begin{equation}
    \tau(\lambda,k) = k = j + y_1m + y_2m',
\end{equation}
and thus under the map above we have 
$g^{-y_1 m}\otimes V_{\tau(\lambda,k)}^0(\lambda) \xto{\sim} V^0_j(\lambda)$.

To see that $\xi_V$ is injective, we can use its $G$-equivariance to reduce to considering
$t, k, k', \lambda$ such that
\begin{equation}\label{eq:xiV}
    \xi_V(g^t \otimes V^0_{\tau(\lambda, k)}(\lambda) ) = \xi_V(1 \otimes V^0_{\tau(\lambda, k')}(\lambda)),\quad \text{or}\quad V^t_{\tau(\lambda, k)+t}(\lambda) = V^0_{\tau(\lambda, k')}.
\end{equation}
We want to show $t\equiv 0\mod \ell$ and $\tau(\lambda,k)=\tau(\lambda, k')$.
From the right hand side of \eqref{eq:xiV} we have  $t\equiv 0 \mod m$ and $\tau(\lambda, k) + t \equiv \tau(\lambda, k') \mod m'$.  Combining these, we find that there exist $z'_1, z'_2 \in \ZZ$ such that
\begin{equation}\label{eq:tauktaukprime}
\tau(\lambda, k)-\tau(\lambda, k') = z_1' m +z_2' m'.
\end{equation}
so $\tau(\lambda, k)-\tau(\lambda, k')$ is a multiple of $d$.
When $q$ is not a root of unity, the definition of $\tau$ implies that in fact $\tau(\lambda, k)-\tau(\lambda, k')=0$ and $k'=k$.
When $q$ is a root of unity, we have $\lambda = q^{\ell\tau(\lambda,k)}\lambda_0 =q^{\ell\tau(\lambda,k')}\lambda_0$
so $|q^\ell|$ divides $\tau(\lambda, k)-\tau(\lambda, k')$ as well, 
which shows $\tau(\lambda, k)-\tau(\lambda, k')$ is divisible by $\lcm(|q^\ell|,d)$.  
Then the restriction that $0 \leq \tau(\lambda,k),\, \tau(\lambda,k') < \lcm(|q^\ell|,d)$ forces $\tau(\lambda,k)=\tau(\lambda,k')$
in this case as well, and thus $t \equiv 0 \mod m'$.  This then implies that $t$ is a multiple of $\ell$, so that 
$g^t \otimes V^0_{\tau(\lambda, k)}(\lambda)=1 \otimes V^0_{\tau(\lambda, k)}=1 \otimes V^0_{\tau(\lambda, k')}$, completing the proof that $\xi_V$ is injective.
It is then straightforward to verify $\xi_V$ is natural in $V$, giving an isomorphism of functors $\Psi\circ \Phi \simeq Id_\cB$.

To show the other composition is the identity, we take
$W \in \cN$, and examine the vector space assigned to vertex $(\lambda,k)$ by the representation $\Phi(\Psi(W))=\Phi(\kk G \otimes_H \tilde{W})$.
As defined in \eqref{eq:Wlambda}, this is the generalized $\lambda$-eigenspace of $e_0 (\kk G \otimes_H \tilde{W}) e_{\tau(\lambda, k)}$.  We claim that this is just $1 \otimes W_{\lambda, k} \cong W_{\lambda, k}$.
Indeed, if $g^t \otimes W_{\lambda',k'}$ were also a summand of the generalized $\lambda$-eigenspace of $e_0 (\kk G \otimes_H \tilde{W}) e_{\tau(\lambda, k)}$, then we would have $\lambda'=\lambda$, $t=0\mod m$,
and $\tau(\lambda, k')+t=\tau(\lambda, k) \mod m'$.
This gives
the same equation \eqref{eq:tauktaukprime} as above, and proceeding by the exact same argument we arrive at the conclusion
that $(\lambda',k')=(\lambda,k) \in \cQ$, and furthermore $t$ is divisible by both $m$ and $m'$, thus divisible by $\ell$,
so $g^t \otimes W_{\lambda',k'} = 1 \otimes W_{\lambda, k}$. This shows that $\Phi(\Psi(W))$ assigns the same vector space to each vertex of $\cQ$ as $W$.

Now it follows from \eqref{eq:xgsw} that the map over the arrow $a_{\lambda, k}$ in $\Phi(\Psi(W))$ is $1\otimes A_{\lambda,k}$,
and it follows from \eqref{eq:glonW} that the map over the arrow $b_{\lambda, k}$ in $\Phi(\Psi(W))$ is $1\otimes B_{\lambda,k}$.  This shows that $\Phi(\Psi(W)) \cong W$ in $\cN$, and from there it is easy to verify the remaining details to see we have an isomorphism of functors $\Phi \circ \Psi \cong Id_\cN$.
\end{proof}

Since the algebra $\Gamma(q^\ell, d)$, and thus $\cN$, only depends on $|q^\ell|$ and $d$, we get the following corollary.

\begin{corollary}\label{cor:independent}
The bimodule category $\cB$ is (up to equivalence of categories) independent of the specific actions of $\uqb$ on $S_m$ and $S_{m'}$.  It only depends on $|q^\ell|$ and $d$, up to equivalence.
\end{corollary}

\begin{remark}
This shows that it will not be feasible to give any kind of ``list'' classifying indecomposable $S_m$-$S_{m'}$ bimodules in $\rep(\uqb)$, since the associated algebra $\Gamma(q^\ell, d)$ will always have finite-dimensional quotients of wild representation type. Informally, this means that the category $\cN$ will always be ``at least as complicated as'' the category of finite-dimensional modules over the free associative $\kk$-algebra $\kk\langle x_1,x_2\rangle$.
\end{remark}

\subsection{Taft algebras and work of Etingof-Ostrik}\label{sec:EO}
Assume here that $\kk$ is algebraically closed, and $q$ is a primitive $n^{\text{th}}$ root of unity.
In this case, Etingof and Ostrik gave a classification of exact module categories over $\rep(T(n))$ by classifying algebras in $\rep(T(n))$ which have no nontrivial right ideals in $\rep(T(n))$ \cite[Theorem 4.10]{EO}.
We begin with a dictionary between the classification of $T(n)$-actions on commutative semisimple algebras $\kk Q_0$ in Corollary \ref{cor:uqbonvertices2} (coming from \cite[Prop.~3.5]{KW16}), and the Etingof-Ostrik classification.  This is logically independent from the rest of the paper, included for completeness of the story.
We then present our results on bimodule categories over generalized Taft algebras.

 Etingof and Ostrik found that the following list covers all algebras in the tensor category $\rep(T(n))$, up to Morita equivalence as algebras in $\rep(T(n))$.  We only describe the algebras enough to establish notation that uniquely identifies them, referring the reader to the proof of \cite[Theorem 4.10]{EO} for more detail, where they proceed by considering the socle filtration of an algebra (as an object of $\rep(T(n))$).
 
 Recall that $G(T(n))=\langle g\rangle$ is cyclic of order $n$. Below, $t$ ranges over the positive divisors of $n$, and $H \leq G$ denotes the unique subgroup of order $t$.
\begin{itemize}
\item[$A(t)$:]
denote by $A(t)$ the algebra of $\kk$-valued functions on the (discrete) group $G/H$.
This is an exact algebra in $\rep(T(n))$ with $x$ acting trivially.  This commutative, semisimple $\kk$-algebra is non-semisimple and indecomposable as an algebra in $\rep(T(n))$.
For $s \in G/H$, we let $e_s$ be the characteristic function of the set $\{s\}$.

\item[$A(t,\lambda)$:] Taking an additional parameter $\lambda \in \kk$, we get an algebra $A(t,\lambda)$
generated by $A(t)$ and an element $y$ satisfying
\begin{equation}
    y = \sum_{s \in G/H} e_{gs}ye_s, \qquad y^n=\lambda.
\end{equation}
These $\kk$-algebras are commutative if and only if $t=n$ (studied earlier in \cite{MS}), and they are all semisimple and indecomposable as algebras in $\rep(T(n))$.
\end{itemize}

We can now compare these to the algebras described in Proposition \ref{prop:uqbonvertices} and Corollary \ref{cor:uqbonvertices2}.
Such an algebra is indecomposable if and only if $G$ acts transitively on the vertices.
Up to relabeling, we have  such a $G$-action on $t$ vertices for each positive divisor $t$ of $n$. 
The following translation can be checked by direct computation.

\begin{proposition}
We have the following correspondence between indecomposable exact algebras in $\rep(T(n))$ described in Proposition \ref{prop:uqbonvertices} and Corollary \ref{cor:uqbonvertices2}, and those in \cite[Theorem 4.10]{EO}:
\begin{itemize}
    \item the unique indecomposable algebra $\kk Q_0$ in our work with $|Q_0|=t$ vertices and parameters $\{\gamma_i\}_{i \in Q_0}=\{0\}$ is $A(n/t)$;
    \item each indecomposable algebra $\kk Q_0$ in our work with $|Q_0|=n$ and parameter set $\{\gamma_i\}_{i \in Q_0}$ is 
    $A(n,\gamma^{-n}(1-q^{-1})^{-n})$, where $\gamma=\gamma_i$ for any choice of $i \in Q_0$.
\end{itemize}
\end{proposition}

We now combine the results of previous sections to describe the bimodule categories over generalized Taft algebras in terms of quiver representations.

\begin{notation}\label{not:T(r,n)} We fix the following notation.
\begin{itemize}
\item $q$ is a primitive $r^{\text{th}}$ root of unity
\item if $m =r$ or $m'=r$, $u$ is the integer so that $r = ud$
\item $\cB_T$ is the full subcategory of $\cB$ consisting of bimodules whose $\uqb$ action factors through the generalized Taft algebra $T(r,n)$
\item $\cT$ is the quiver with vertices $(\zeta, i)$, where $\zeta$ is an $(n/\ell)^{\text{th}}$ root of 1 and $i \in \ZZ/d\ZZ$, and $a$-type arrows $(\zeta,i+1)\to(\zeta,i)$
\item $\Gamma_T$ is the quotient of $\kk \cT$ by the ideals generated by the following relations
\begin{equation}\label{eq:arelations}
\begin{cases}
a^r= 0 & m \neq r, m'\neq r \\
a^u = \zeta^{-z_1}\left((\gamma_0)^r\zeta - (\gamma'_0)^r \right) & m = m' = r\\
a^u = \zeta^{z_2}\gamma_0^r & m =r, \; m' \neq r \\
a^u = (-1)^{r-d+1} \zeta^{-z_1}(\gamma_0')^r & m \neq r, \; m' = r
\end{cases}
\end{equation}
\end{itemize}
\end{notation}

Note that the connected components of $\cT$ have the form of \eqref{eq:Sr} without the loops on each vertex. 

Let $V \in \cB_T$ and $\Phi(V) = W = (W_{\lambda,i}, A_{\lambda,i}, B_{\lambda,i})$. Note that $m$ and $m'$ must divide $n$ by Corollary \ref{cor:uqbonvertices2}, so $g^\ell$ has $(n/\ell)^{\text{th}}$ roots of unity as eigenvalues and is diagonalizable. Therefore, the $B_{\lambda,i}$ maps are 0. At this point, we separate our analysis into two cases.

\subsubsection*{Case 1: $m\neq r$ and $m'\neq r$}
In this case, Corollary \ref{cor:uqbonvertices2} implies that $\gamma_0 = 0$ and $\gamma_0' = 0$. 
Since the $A_{\lambda,i}$ maps are restrictions of $\sigma$, it follows from Corollary \ref{cor:uqbonkQ2} that $\Phi(V)$ is a representation of $\Gamma_T$. Note that in this case, $\Gamma_T$ is the quotient of $\kk \cT$ by $\left(\kk \cT_1\right)^r=\rad^r\left(\kk\cT\right)$.

\subsubsection*{Case 2: $m =r$ or $m'=r$} In this case, $d$ divides $r$ and $r$ divides $\ell$, so $q^\ell = 1$, and $\cT$ has $n/\ell$ connected components of length $d$. A connected component of $\cT$ is shown below where $\zeta$ is an $(n/\ell)^{\text{th}}$ root of unity.

\begin{equation}
\vcenter{\hbox{
\begin{tikzpicture}[point/.style={shape=circle,fill=black,scale=.5pt,outer sep=3pt},>=latex]
\node[point,label={-100:$(\zeta,d-1)$}] (0) at (0,0) {};
\node[point,label={below:$(\zeta,d-2)$}] (1) at (3,0) {};
\node (dots) at (6,0) {$\cdots$};
\node[point,label={-75:$(\zeta,0)$}] (d) at (9,0) {};
\path[->] (0) edge node[above] {$a_{\zeta,d-2}$} (1); 
\path[->] (1) edge node[above] {$a_{\zeta,d-3}$} (dots);
\path[->] (dots) edge node[above] {$a_{\zeta,0}$} (d) ;
\path[->] (d) edge [bend left=30] node[midway, below] {$a_{\zeta,d-1}$} (0);
\end{tikzpicture} }}
\end{equation}
It follows from Corollary \ref{cor:uqbonkQ2}, and the fact that $kz_1m$ is always a multiple of $\ell$, that $\Phi(V)$ satisfies the relations necessary to be a representation of $\Gamma_T$. 

\begin{remark}\label{rem: r=n}
Suppose we are in Case 2, and that $r=n$, so $T(r,n) = T(n)$ is the $n^{\text{th}}$ Taft algebra. Then the reductions above (and switching $m, m'$ if necessary) imply that the setup simplifies to $d=m\leq m'=n=\ell$.
Thus $q^\ell=1$ and $\cT$ has a single connected component consisting of the vertices $(1,i)$ where $0 \leq i < m$.
\end{remark}

Both cases above result in the following corollary.

\begin{corollary}\label{cor:taftquivers}
The category $\cB_T$ is equivalent to $\rep(\Gamma_T)$.
\end{corollary}

\begin{proof}
We have shown above that $\Phi$ restricts to a functor from $\cB_T$ to $\rep(\Gamma_T)$, so we need to show that each isomorphism class in $\rep(\Gamma_T)$ is in the image of this functor.

Given a representation $W$ of $\Gamma_T$, we can see this a representation in $\cN$ with all $B_{\lambda,i}=0$. Thus, $\Psi(W)$ is in $\cB$. It follows from \eqref{eq:sigmaformula} that given the action of $x$ on $\kk G \otimes_H \tilde{W}$ defined in \eqref{eq:xgsw}, the map $\sigma$ from Theorem \ref{thm:uqbonkQ} is given by
\begin{equation}
\sigma(g^t\otimes w) = q^{-t} g^{t+\epsilon}\otimes A_{\lambda,k}(w)
\end{equation}
for $t \in \ZZ^+$ and $w \in W_{\zeta, i}$. Using this to compute the $r^{\text{th}}$ power of $\sigma$, we see that Corollary \ref{cor:uqbonkQ2} implies that $\Psi(W) \in \cB_T$.
\end{proof}

It is easy to see that $\Gamma_T$ is always finite dimensional over $\kk$.
When the relations defining $\Gamma_T$ only set paths to 0, it is a self-injective Nakayama algebra. 
The indecomposable objects in this situation can easily be described explicitly \cite[Ch.~V]{ASS06}, resulting in the following.

\begin{corollary}
In Case 1 and in each case where the scalar on the right hand side of \eqref{eq:arelations} vanishes,
the bimodule category $\cB_T$ is uniserial.
\end{corollary}

This raises the following natural questions.

\begin{question}
When the scalars on the right hand sides of the equations in \eqref{eq:arelations} are nonzero, is $\rep(\Gamma_T)$ semisimple? 
If not, what is a quiver with \emph{admissible} relations giving a Morita equivalent algebra?
\end{question}

\subsection{Bimodules for $\Uqsl$ and $\uqsl$}
Finally, we extend our results to $\Uqsl$ and $\uqsl$ by gluing over Borel subalgebras the equivalences from previous sections.
We will make some simplifying assumptions to avoid degenerate cases; see Notation \ref{not:Uqsl2} below.  First, we introduce the relevant quivers and algebras.

Fix a root of unity $\mu$, a constant $\eta \in \kk^{\times}$ and a positive integer $N$. Define a quiver $\cQ'(\mu,N)$ with the same vertex set and arrows as $\cQ(\mu,N)$ along with additional $c$-type arrows $(\lambda, k) \to (\mu\lambda,k-1)$ coming out of each vertex $(\lambda, k)$. 

Denote the connected components of this quiver by $\cS'_p$ where $p = \lcm{(\left| \mu \right|, N)}$. 
Each vertex $(\lambda, k)$ has an $a$-type arrow and a $c$-type arrow coming out of it. Then, $(\lambda, k)$ is contained in a cycle of the form $a^p$ and another cycle of the form $c^p$.
It follows that $\cS'_p$ has $p^2$ vertices with $p$ cycles consisting of $a$-type arrows and $p$ cycles consisting of $c$-type arrows.

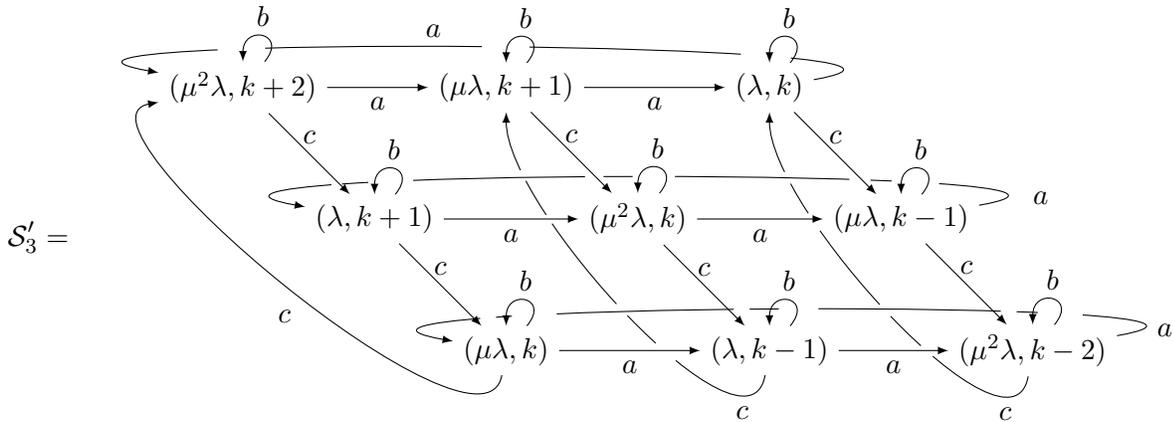
\tikzset{every loop/.style={distance=5mm,in=90,out=50}}

\begin{equation}\label{eq:S'3}
\cS'_3 = \hspace*{-2cm}
\vcenter{\hbox{
\begin{tikzpicture}[point/.style={shape=circle,fill=black,scale=.5pt,outer sep=3pt},>=latex]

\node (0-2) at (-2, 1.75) {$(\mu^2\lambda,k+2)$};
\node (1-1) at (1.5, 1.75) {$(\mu\lambda, k+1)$};
\node (20) at (5, 1.75) {$(\lambda, k)$};

\node (-1-1) at (-0.25,0) {$(\lambda,k+1)$};
\node (00) at (3.25,0) {$(\mu^2\lambda,k)$};
\node (11) at (6.75,0) {$(\mu\lambda,k-1)$};

\node (-20) at (1.5,-1.75) {$(\mu\lambda,k)$};
\node (-11) at (5,-1.75) {$(\lambda,k-1)$};
\node (02) at (8.5,-1.75) {$(\mu^2\lambda,k-2)$};

\path[->] (0-2) edge[loop] node[above] {$b$} (0-2);
\path[->] (1-1) edge[loop] node[above] {$b$} (1-1);
\path[->] (20) edge[loop] node[above] {$b$} (20);
\path[->] (-1-1) edge[loop] node[above] {$b$} (-1-1);
\path[->] (00) edge[loop] node[above] {$b$} (00);
\path[->] (11) edge[loop] node[above] {$b$} (11);
\path[->] (-20) edge[loop] node[above] {$b$} (-20);
\path[->] (-11) edge[loop] node[above] {$b$} (-11);
\path[->] (02) edge[loop] node[above] {$b$} (02);

\path[->] (-1-1) edge node[below] {$a$} (00);
\path[->] (00) edge node[below] {$a$} (11) ;
\path[->] (-20) edge node[below] {$a$} (-11);
\path[->] (-11) edge node[below] {$a$} (02) ;
\path[->] (0-2) edge node[below] {$a$} (1-1);
\path[->] (1-1) edge node[below] {$a$} (20);

\path[->] (20) edge[out=10, in = 170, distance = 30mm, dash pattern=on 33pt off 19pt on 80pt off 18pt on 82 pt off 18 pt on 75 pt] node[above, xshift=-7mm] {$a$} (0-2);

\path[->] (11) edge[out=10, in = 170, distance = 30mm, dash pattern=on 39pt off 18pt on 8pt off 8pt on 66pt off 16pt on 11pt off 7pt on 65pt off 17pt on 10pt off 4pt on 70pt] node[above, xshift=53mm, yshift=-5mm] {$a$} (-1-1);
\path[->] (02) edge[out=10, in = 170, distance = 30mm, dash pattern=on 37pt off 17pt on 8pt off 6pt on 69pt off 15pt on 11pt off 6pt on 67pt off 17pt on 6pt off 6pt on 70pt] node[above, xshift=50.5mm, yshift=-5mm] {$a$} (-20);

\path[->] (0-2) edge node[above] {$c$} (-1-1);
\path[->] (1-1) edge node[above] {$c$} (00);
\path[->] (20) edge node[above] {$c$} (11);
\path[->] (-1-1) edge node[above] {$c$} (-20);
\path[->] (00) edge node[above] {$c$} (-11);
\path[->] (11) edge node[above] {$c$} (02);
\path[->] (-20) edge[out=260, in=190, distance=14mm] node[below, yshift=-2mm] {$c$} (0-2);
\path[->] (-11) edge[out=260, in=270, dash pattern=on 42pt off 10pt on 15pt off 6pt on 39pt off 4pt on 15pt off 4.5pt on 25pt, distance=13mm] node[above,xshift=15mm, yshift=-15mm] {$c$} (1-1);
\path[->] (02) edge[out=260, in=270, distance=13mm, dash pattern=on 45pt off 6pt on 16pt off 6pt on 36pt off 8pt on 13pt off 5pt on 50pt] node[above,xshift=15mm, yshift=-15mm] {$c$} (20);
\end{tikzpicture}
}}
\end{equation}

We assume that $m,m'>2$ to avoid degenerate cases which require some special treatment. By Corollary \ref{cor:uqslonvertices1}, this puts us in the case where $q$ is a root of unity.
For simplicity, we restrict to the case that $q$ is a primitive $n^{\rm th}$ root of unity with $n>2$ and $n$ odd,
which by Corollary \ref{cor:uqslonvertices2} forces that both $m, m'$ are divisible by $n$.
Furthermore, this implies that the hypothesis of Remark \ref{rem:orbitsgt2} is satisfied.

\begin{notation}\label{not:Uqsl2}
We fix the following notation.
\begin{itemize}
    \item $\Gamma'(\mu,\eta,N)$ is the quotient of $\kk \cQ'(\mu,N)$ by the ideal generated by all relations of the form
\begin{equation}\label{eq:Q'relations}
ab = \mu^{-1}ba, \qquad
cb = \mu bc, \qquad 
\eta ac = ca.
\end{equation}
    \item $\cN' :=\nrep \Gamma'(q^{2\ell}, q^2,d)$
    \item $S_m$ and $S_{m'}$ are defined the same way as in Section \ref{sec:quiversbimods} except we assume these algebras have fixed $\Uqsl$-actions extending the $G$-actions
    \item $\cB'$ is the category of finite dimensional $S_m$-$S_{m'}$-bimodules in $\rep{(\Uqsl)}$
\end{itemize}
\end{notation}

Denote elements of $\cN'$ by $W = (W_{\lambda,k},A_{\lambda,k},C_{\lambda,k}, B_{\lambda,k})$ where
\begin{equation}
    A_{\lambda,k} \in \Hom_\kk(W_{\lambda,k}, W_{\mu\lambda,k+1}), \qquad
    C_{\lambda,k} \in \Hom_\kk(W_{\lambda,k}, W_{\mu^{-1}\lambda,k+1}), \qquad
    B_{\lambda,k} \in \End_\kk(W_{\lambda,k}).
\end{equation}

\begin{theorem} \label{thm:uqslequivalence}
There exist mutually quasi-inverse functors
\begin{equation}
\Phi': \cB' \to \cN' \qquad and \qquad \Psi': \cN' \to \cB',
\end{equation}
and therefore the categories $\cB'$ and $\cN'$ are equivalent.
\end{theorem}

\begin{proof}
Let $G=\langle K\rangle$ and fix a bimodule $V \in \cB'$. The bimodule $V$ is a representation of both the Borel subalgebras $U_{q^2}(\fb)$ and $U_{q^{-2}}(\fb)$ as in \eqref{eq:borels}. 
By Proposition \ref{prop:Phi}, we have a functor $\Phi_E$ so that $\Phi_E(V)$ is a representation of $\Gamma(q^{2\ell},d)$ denoted $(W_{\lambda,k}, A_{\lambda,k}, B_{\lambda,k})$ and a functor $\Phi_F$ so that $\Phi_F(V)$ is a representation of $\Gamma(q^{-2\ell},d)$ denoted $(W_{\lambda,k}, C_{\lambda,k},B_{\lambda,k})$. 
It follows from \eqref{eq:sigmacommutativity} that  
\begin{equation}
C_{q^{2\ell}\lambda,k+1}A_{\lambda,k} = q^2A_{q^{-2\ell}\lambda,k+1}C_{\lambda,k},
\end{equation}
and therefore $(W_{\lambda,k}, A_{\lambda,k}, C_{\lambda,k}, B_{\lambda,k})\in \cN'$. Let $\Phi'(V) = (W_{\lambda,k}, A_{\lambda,k}, C_{\lambda,k}, B_{\lambda,k})$.

Notice that $\Gamma'(q^{2\ell}, q^2,d)$ contains subalgebras $\Gamma(q^{2\ell},d)$ and $\Gamma(q^{-2\ell,d})$ generated by the sets of arrows $\left\{a_{\lambda,k}, b_{\lambda,k}\right\}$ and $\left\{b_{\lambda,k}, c_{\lambda,}\right\}$ respectively. 
For a representation $W = (W_{\lambda,k}, A_{\lambda,k}, C_{\lambda,k}, B_{\lambda,k})$ in $\cN'$, the subrepresentation $W' = (W_{\lambda,k}, A_{\lambda,k}, B_{\lambda,k})$ is a representation of $\Gamma(q^{2\ell},d)$ and the subrepresentation $W'' = (W_{\lambda,k}, C_{\lambda,k}, B_{\lambda,k})$ is a representation of $\Gamma(q^{-2\ell,d})$.
We can construct the space $\tilde{W}$ from the spaces $W_{\lambda,k}$ as in \eqref{eq:Wbimod} and \eqref{eq:Wtilde}. 

By Proposition \ref{prop:Psi}, there exists a functor $\Psi_E$ so that $\Psi_E(W') = \kk G \otimes_H \tilde{W}$ is an $S_m$-$S_{m'}$-bimodule  in $\rep{(U_{q^2}(\fb))}$ and there exists a functor $\Psi_F$ so that $\Psi(W'')= \kk G \otimes_H \tilde{W}$ is an $S_m$-$S_{m'}$-bimodule in  $\rep{(U_{q^{-2}}(\fb))}$. 
By construction of the two functors $\Phi_E$ and $\Phi_F$, the $S_m$-$S_{m'}$-bimodule structures on $\Psi_E(W')$ and $\Psi_F(W'')$ are identical. The actions of $K$ on $\kk G \otimes \tilde{W}$ given by $\Psi_E$ and $\Psi_F$ are also identical since they are induced by the same set of maps $B_{\lambda,k}$.

Let $\{\gamma_{i}^E\}_{i=0}^{m-1}$ and $\{\gamma_{j}^{E'}\}_{j =0}^{m'-1}$ be the scalars from Proposition \ref{prop:uqslonvertices} determining the actions of the generator $E$ of $\Uqsl$ on $S_m$ and $S_{m'}$ respectively.  
For $w \in W_{\lambda,k}$ and $t \in \ZZ$, letting $j = \tau(\lambda,k)$, we define an action of $E$ on $\kk G \otimes \tilde{W}$ by
\begin{equation}\label{eq:EKsw}
E\cdot (K^t\otimes w) = 
q^{-t}\left(\gamma^{E'}_j K^t\otimes w - \gamma_0^Eq^{-1}K^{t+1}\otimes w + K^{t+\epsilon}\otimes A_{\lambda,k}(w)\right).
\end{equation}
Similarly, let $\{\gamma_{i}^F\}_{i=0}^{m-1}$ and $\{\gamma_j^{F'}\}_{j =0}^{m'-1}$ be the scalars determining the action of the generator $F$ of $\Uqsl$ on $S_m$ and $S_{m'}$ respectively. 
We define an action of $F$ on $\kk G \otimes \tilde{W}$ by
\begin{equation}\label{eq:FKsw}
F\cdot (K^t \otimes w) = q^{-t}\left(\gamma^{F'}_j K^{t-1}\otimes w - \gamma_0^Fq^{-1}K^{t}\otimes w + K^{t-1+\epsilon}\otimes C_{\lambda,k}(w)\right).
\end{equation}
Notice that the actions defined in \eqref{eq:EKsw} and \eqref{eq:FKsw} are exactly the actions of the generators $x$ of the two Borel subalgebras of $\Uqsl$ on $\kk G \otimes \tilde{W}$ given by the functors $\Psi_E$ and $\Psi_F$ under the identifications in \eqref{eq:borels}.  One can check that these action satisfy relations \eqref{eq:uqslrelations}, so $\kk G \otimes_H \tilde{W} \in \cB'$. Let $\Psi'(W) = \kk G \otimes_H \tilde{W}$.

The fact that $\Phi'$ and $\Psi'$ are mutually quasi-inverse follows directly from Theorem \ref{thm:uqbequivalence} which gives us that $\Phi_E, \Psi_E$ and $\Phi_F, \Psi_F$ are both pairs of mutually quasi-inverse functors.
\end{proof}

Now, we consider the bimodules in $\cB'$ whose $\Uqsl$ actions factor through $\uqsl$. 
From Corollary \ref{cor:uqslonvertices1} and our simplifying assumption that $m, m'>2$, we are reduced to the case $m=m'=n$.
Then this makes $\ell=d=n$ and thus $q^\ell=1$.

\begin{notation}\label{not:uqsl2} We fix the following notation
\begin{itemize}
    \item $\cB'_T$ is the full subcategory of bimodules in $\cB'$ whose $\Uqsl$ actions factor through $\uqsl$.
    \item $\cT'$ is the quiver
\begin{equation}\label{eq:T'}
\cT' = 
\vcenter{\hbox{
\begin{tikzpicture}[point/.style={shape=circle,fill=black,scale=.5pt,outer sep=3pt},>=latex]
\node[point,label={-100:$(1,d-1)$}] (0) at (0,0) {};
\node[point,label={below:$(1,d-2)$}] (1) at (3,0) {};
\node (dots) at (6,0) {$\cdots$};
\node[point,label={-75:$(1,0)$}] (d) at (9,0) {};
\path[->] (0.40) edge node[above] {$a$} (1.130); 
\path[->] (1.40) edge node[above] {$a$} (dots.166.25);
\path[->] (dots.14.75) edge node[above] {$a$} (d.130) ;
\path[->] (d) edge [out=245, in=295, distance=20mm] node[midway, above] {$a$} (0);
\path[->] (0.320) edge node[below] {$c$} (1.210); 
\path[->] (1.320) edge node[below] {$c$} (dots.192.5);
\path[->] (dots.347.5) edge node[below] {$c$} (d.210) ;
\path[->] (d) edge [out=280, in=260, distance=21mm] node[midway, below] {$c$} (0);
\end{tikzpicture} }}
\end{equation}
\item $\Gamma_T'$ is the quotient of $\kk \cT'$ generated by all relations of the form
\begin{equation}
\begin{gathered}
q^2ac = ca, \qquad 
a^d = (\gamma_0^E)^n - (\gamma_0^{E'})^n, \\ \text{and} \qquad 
c^d = (\gamma_0^F)^n - (\gamma_0^{F'})^n.
\end{gathered}
\end{equation}

\end{itemize}
\end{notation}

\begin{corollary}
The categories $\cB'_T$ and $\rep(\Gamma_T')$ are equivalent.
\end{corollary}
\begin{proof}
Recall that the two Borel subalgebras of $\uqsl$ are Taft algebras. Since we are assuming $m$ and $m'$ are greater than 2, it follows from Corollary \ref{cor:uqslonvertices2} that $m = m' = r = n$, and Remark \ref{rem: r=n} applies to these Borel subalgebras.
The statement now follows from Theorem \ref{thm:uqslequivalence} and Corollary \ref{cor:taftquivers}.
\end{proof}

A standard reordering argument shows that $\Gamma'_T$ is always finite dimensional over $\kk$.
This raises the following natural questions.
\begin{question}
When either or both of $(\gamma_0^E)^n - (\gamma_0^{E'})^n$ and 
$(\gamma_0^F)^n - (\gamma_0^{F'})^n$ are nonzero, what is a quiver with \emph{admissible} relations which is Morita equivalent to $\Gamma_T'$?  In any case, is $\Gamma_T'$ of finite, tame, or wild representation type?
\end{question}

\bibliographystyle{alpha}
\bibliography{Uqsl2actions}

\end{document}